\documentclass[11pt,oneside,english]{amsart}
\usepackage[T1]{fontenc}
\usepackage[utf8]{inputenc}
\usepackage{color}
\usepackage{babel}
\usepackage{amstext}
\usepackage{amsthm}
\usepackage{amssymb}
\usepackage{graphicx}
\usepackage[unicode=true,pdfusetitle,
 bookmarks=true,bookmarksnumbered=false,bookmarksopen=false,
 breaklinks=false,pdfborder={0 0 0},backref=false,colorlinks=true]
 {hyperref}
\hypersetup{
 citecolor=blue}

\makeatletter
\numberwithin{equation}{section}
\numberwithin{figure}{section}
\theoremstyle{plain}
\newtheorem*{cor*}{\protect\corollaryname}
\theoremstyle{plain}
\newtheorem{thm}{\protect\theoremname}[section]
\theoremstyle{definition}
\newtheorem{defn}[thm]{\protect\definitionname}
\theoremstyle{remark}
\newtheorem{rem}[thm]{\protect\remarkname}
\theoremstyle{plain}
\newtheorem{prop}[thm]{\protect\propositionname}
\theoremstyle{plain}
\newtheorem{lem}[thm]{\protect\lemmaname}
\theoremstyle{plain}
\newtheorem{cor}[thm]{\protect\corollaryname}

\usepackage{geometry}
\usepackage{amsmath}
\usepackage{amsthm}

\numberwithin{equation}{section}
\numberwithin{figure}{section}
\usepackage{enumitem}		
 \let\footnote=\endnote
\@ifundefined{lettrine}{\usepackage{lettrine}}{}

\theoremstyle{definition}
\newtheorem{thmx}{Theorem}

\def\d{\delta}

\def\G{\Gamma}

\def\g{\gamma}
\def\l{\lambda}

\def\D{\mathbb{D}}
\def\R{\mathbb{R}}

\def\N{\mathbb{N}}

\def\ep{\varepsilon}
\def\vphi{\varphi}

\def\cal{\mathcal}

\def\T{\mathcal{T}}
\def\F{\mathcal{F}}

\def\lt{\lambda_{T}}

\usepackage{float}

\address{Department of Mathematics,
The Ohio State University, Columbus, OH 43210, USA} 
\email{chen.8022@osu.edu}
\address{Department of Mathematics, The University of Chicago, Chicago, IL 60637, USA} 
\email{lkao@math.uchicago.edu}
\address{Department of Mathematics, The University of Chicago, Chicago, IL 60637, USA} 
\email{kihopark@math.uchicago.edu}


\subjclass[2000]{}
\thanks{The second author gratefully acknowledges support from the National Science Foundation Postdoctoral Research Fellowship under Grant Number DMS 1703554.}

\usepackage[normalem]{ulem}

\def\G{\Gamma}
\def\g{\gamma}
\def\la{\lambda}

\def\D{\Delta}
\def\R{\mathbb{R}}

\def\T{\mathcal{T}}
\def\CC{\mathcal{C}}

\def\ep{\varepsilon}

\def\T{\mathcal{T}}
\def\F{\mathcal{F}}

\def\lt{\lambda_{T}}

\def\Sing{\mathrm{Sing}}
\def\Reg{\mathrm{Reg}}
\def\vp{\varphi}
\def\hol{H\"older }
\def\wt{\widetilde}

\makeatother

  \providecommand{\corollaryname}{Corollary}
  \providecommand{\definitionname}{Definition}
  \providecommand{\lemmaname}{Lemma}
  \providecommand{\propositionname}{Proposition}
  \providecommand{\remarkname}{Remark}
  \providecommand{\theoremname}{Theorem}
\providecommand{\theoremname}{Theorem}

\begin{document}

\title[Properties of equilibrium states for manifolds without focal points]{Properties of equilibrium states for geodesic flows over manifolds without focal points}

\author{Dong Chen, Lien-Yung Kao, and Kiho Park}

\date{\today}

\maketitle

\begin{abstract}
We prove that for closed rank 1 manifolds without focal points the equilibrium states are unique for H\"older potentials satisfying the pressure gap condition. In addition, we provide a criterion for a continuous potential to satisfy the pressure gap condition. Moreover, we derive several ergodic properties of the unique equilibrium states including the equidistribution and the $K$-property.
\end{abstract}


\section{Introduction}

Let $M$ be a closed rank 1 Riemannian manifold without focal points,
and let $\mathcal{F}=\{f_{t}\}_{t\in\R}$ be the geodesic flow on
the unit tangent bundle $T^{1}M$. Manifolds without
focal points are natural generalizations of nonpositively curved manifolds.
These two classes of manifolds share many geometric and dynamic
features. For example, the flat strip theorem holds for both classes of manifolds. As a result, the subset of $T^1M$ that does not display hyperbolic behavior with respect to the geodesic flow can be described in the same manner for these two classes of manifolds. Such subset is called the singular set, and denoted by $\Sing$. Nevertheless, lack of certain geometric properties, such as the convexity of the
Jacobi fields, impose further difficulty in analyzing the geodesic flows over manifolds without focal points; see Remark \ref{rem: lt vs lambda 1} and \ref{rem: lt vs lambda 2} for details.

In \cite{chen2018unique}, the authors were able to partially
overcome the difficulties and generalize the results of Burns, Climenhaga, Fisher, and Thompson \cite{burns2018unique} on the properties of the equilibrium states for rank 1 nonpositively curved manifolds to surfaces without focal points. In this work, we extend most results of \cite{chen2018unique} to rank 1 manifolds without focal points of arbitrary dimension. Manifolds without focal points are defined in Definition \ref{defn: no focal}, and rank 1 condition on $M$ means that there exists at least one rank 1 vector in $T^1M$; see Definition \ref{defn: rank}.

To put our results in context, we say a continuous potential
$\vphi\colon T^1M\to\R$ satisfies the \textit{pressure gap condition} if 
\begin{equation}
P(\vphi)>P(\Sing,\vphi).\label{eq: pressure gap}
\end{equation}
Note that $\Sing$ is closed and $\F$-invariant, and $P(\Sing,\vphi)$ refers to the pressure of $\vphi$ restricted to $\Sing$. 
\begin{thmx}\label{thm: 1} Let $M$ be a closed rank 1 Riemannian
manifold without focal points, $\mathcal{F}$ be the geodesic flow
over $M$, and $\vphi\colon T^{1}M\to\R$ be a \hol potential. If
$\vphi$ satisfies the pressure gap condition, then $\vphi$ has a unique equilibrium state $\mu_{\vphi}$. \end{thmx}

The origin of Theorem \ref{thm: 1} traces back to the work of Bowen
\cite{bowen1974some} where he showed that the bounded distortion property (also known as the Bowen property) on the potential and the expansivity and the specification property on the dynamical system guarantee the existence of a unique equilibrium state. Natural examples of such systems are uniformly hyperbolic systems. Since the work of Bowen, his result has been extended in various directions, and
one such direction aims at relaxing the assumptions on the base dynamical systems. Recently, Climenhaga and Thompson \cite{climenhaga2016unique}
developed a set of criteria that guarantees the existence of a unique
equilibrium state which applies to many non-uniformly hyperbolic
systems; see \cite{Climenhaga:2015wf}, \cite{burns2018unique}, \cite{chen2018unique}, and \cite{climenhaga2019equilibrium}. 

Among non-uniformly hyperbolic dynamical systems arising from
geometry, the first result in the same flavor as Theorem \ref{thm: 1}
was the work of Knieper \cite{knieper1998uniqueness}. He employed
Patterson-Sullivan theory to establish the uniqueness of the measure of maximal entropy for geodesic flows over nonpositively curved manifolds.
Twenty years later, Burns, Climenhaga, Fisher, and Thompson \cite{burns2018unique} extended
Knieper's result to equilibrium states for potentials with the pressure
gap. Their approach is inspired by the previously mentioned work of Bowen \cite{bowen1974some}. Previous work by the authors \cite{chen2018unique} used the same approach to further generalize this result to surfaces without focal points.
Finally, in this work, Theorem \ref{thm: 1} shows
that the same philosophy holds for higher dimensional manifolds without
focal points. 

From Theorem \ref{thm: 1}, it automatically follows that these equilibrium
states are ergodic. It is natural to ask whether such equilibrium states possess stronger ergodic properties. Indeed, for uniformly hyperbolic systems, the unique equilibrium states for \hol potentials have many stronger ergodic properties; these include being Bernoulli, and having equidistribution property by weighted periodic orbits as well as statistical properties such as the central
limit theorem and the large deviation property; see \cite{Parry:1990tn} and \cite{katok1997introduction}.

The following theorem partially answers the question above, 
and establishes several ergodic properties of the unique equilibrium
state $\mu_{\vphi}$ from Theorem \ref{thm: 1}.

\begin{thmx}\label{thm: 2} In the setting of Theorem \ref{thm: 1},
the unique equilibrium state $\mu_{\vphi}$ has the following properties: $\mu_{\vphi}$ has the $K$-property and
is fully supported. Moreover, $\mu_\vphi$ is equal to the weak-$*$ limit of the weighted regular closed geodesics, and $\mu_{\vphi}(\Reg)=1$. \end{thmx}

Notice that for surfaces without focal points, the unique equilibrium state $\mu_\vphi$ has the Bernoulli property; see \cite[Theorem B]{chen2018unique}. Recall that Bernoulli is the strongest ergodic property
that implies the $K$-property, and the $K$-property then implies mixing.
We refer the readers to Section \ref{sec: proof of thm B} for more
details on the properties listed in Theorem \ref{thm: 2}. 

Next theorem extends \cite[Theorem B]{burns2018unique} and establishes a criterion on $\vphi$ that guarantees the pressure gap property.

\begin{thmx}\label{thm: 3} With $M$ and $\mathcal{F}$ as in Theorem
\ref{thm: 1}, suppose $\vphi\colon T^{1}M\to\R$ is a continuous
potential that is locally constant on a neighborhood of $\Sing$.
Then $\vphi$ has the pressure gap. \end{thmx}

Since the zero potential $\vphi\equiv0$ trivially satisfies the assumption
in Theorem \ref{thm: 3}, as its corollary
we obtain the uniqueness of the measure of maximal entropy for geodesic flows over manifolds without focal points. Moreover,
such equilibrium states have the $K$-property. These results have also
recently been established by Liu, Liu, and Wang \cite{Liu:2018ud} and Liu, Wang, and Wu \cite{liu2018patterson}
via a different approach based on the previously mentioned work of Knieper \cite{knieper1998uniqueness} as well as Babillot \cite{Babillot:2002bq}: 
\begin{cor}
\ \begin{enumerate}[font=\normalfont] 

\item\cite[Theorem A]{liu2018patterson} Geodesic flows over manifolds without focal points have unique measures of maximal entropy.

\item\cite[Theorem 2.2]{Liu:2018ud} Such measures of
maximal entropy are mixing. 

\end{enumerate}
\end{cor}

We remark that for surfaces without focal points, the uniqueness of
the measures of maximal entropy was first proved by 
Gelfert and Ruggiero \cite{Gelfert:2017tx}. A recently work of Climenhaga,
Knieper, and War \cite{climenhaga2019uniqueness} further extended
this result to geodesic flows over surfaces without conjugate points. 

The paper is organized as follows. In Section \ref{sec: 2}, we introduce our setting of manifolds without focal points. Moreover, we introduce a function to measure hyperbolicity on $T^1M$ and study its properties to be used later in the paper. In Section \ref{sec: 3}, we survey relevant results in thermodynamic formalism, and state the Climenhaga-Thompson criteria that will used to prove Theorem \ref{thm: 1} in Section \ref{sec: proof of thm A}. In Section \ref{sec: proof of thm B}, we establish the ergodic properties of the unique equilibrium states listed in Theorem \ref{thm: 2}. Lastly, in Section \ref{sec: proof for thm 3}, we prove Theorem \ref{thm: 3}.
\\\\
\noindent \textit{Acknowledgments.} The authors would
like to thank Keith Burns, Vaughn Climenhaga, François Ledrappier,
Dan Thompson, and Amie Wilkinson for inspiring discussions and supports.

\section{Geometry}\label{sec: 2}
\subsection{Manifolds with no focal points}

In this subsection, we introduce and survey geometric features of the manifold without focal points. These results can be found in \cite{Eberlein:1973hu,Pesin:1977vi,Eschenburg:1977kn,Burns:1983dw}.

Throughout this section $M$ denotes a closed Riemannian manifold, and we denote the geodesic flow on its unit tangent
bundle $T^{1}M$ by ${\cal F}=(f_{t})_{t\in\R}$. 
Recall that for any Riemannian manifold $M$,
we can naturally equip its tangent bundle $T^1M$ with the \textit{Sasaki
metric}; see \cite{doCarmo:2013tg}. In what follows, without stating specifically, the norm $||\cdot||$ on $TT^1M$ always refers to the Sasaki metric.

A \textit{Jacobi field} $J(t)$ along a geodesic $\g$ is a vector
field along $\g$ satisfying the \textit{Jacobi equation}: 
\begin{equation}\label{eq: jacobi eqn}
\textcolor{black}{J''(t)+R(J(t),\dot{\g}(t))\dot{\g}(t) = 0,}
\end{equation}
where $R$ is the Riemannian curvature tensor, and $'$ denotes the covariant derivative along $\g$.
 
A Jacobi field $J$ is \textit{orthogonal} if both
$J$ and $J'$ are orthogonal to $\dot{\g}$ at some $t_{0}\in\R$
(and hence for all $t\in\R)$. 
A Jacobi field $J$ is \textit{parallel} at $t_0$ if $J'(t_0)=0$. If $J'(t) = 0 $ for all $t \in \R$, then we say $J$ is \textit{parallel}.

\begin{defn}[No focal points]\label{defn: no focal}
A Riemannian manifold $M$ has \textit{no focal
points} if for any initial vanishing Jacobi field $J(t)$, its length
$\left\Vert J(t)\right\Vert $ is strictly increasing. 
\end{defn}



It is a classical result that one can identify the tangent space of
$T^{1}M$ with the space of orthogonal Jacobi fields ${\cal J}$.
Moreover, one can use this relation to define three ${\cal F}-$invariant
bundles $E^{u},E^{c},$ and $E^{s}$ in $TT^{1}M$. To be more precise, 
for each $v \in T^1M$, there exists a direct sum decomposition $T_vT^1M = H_v \oplus V_v$ into the horizontal and vertical subspaces, each equipped with the norm induced from the Riemannian metric on $M$. The \textit{Sasaki metric} on $T^1M$ is defined by declaring $H_v$ and $V_v$ \textcolor{black}{to be orthogonal.}
Denoting the space of orthogonal Jacobi
fields along a geodesic $\g$ by ${\cal J}(\g)$, 
the identification
between $T_{v}T^{1}M$ and ${\cal J}(\g_{v})$ is given by 
\[
T_{v}T^{1}M\ni\xi = (\xi_h,\xi_v)\mapsto J_{\xi}\in{\cal J}(\g_{v})
\]
 where $J_{\xi}$ is the unique Jacobi field characterized by $J_\xi(0) = \xi_h$ and $J'_{\xi}(0) = \xi_v$. Moreover, we have 
\begin{equation}
\left\Vert df_{t}(\xi)\right\Vert ^{2}=\left\Vert J_{\xi}(t)\right\Vert ^{2}+\left\Vert J'_{\xi}(t)\right\Vert ^{2}.\label{eq:metric_jacobi_unit_tangent_bundle}
\end{equation}

We define ${\cal J}^{s}(\g)$ to be the space of \textit{stable (orthogonal)
Jacobi fields} as 
\[
{\cal J}^{s}(\g)=\{J(t)\in{\cal J}(v):\ \left\Vert J(t)\right\Vert \ {\rm is}\ {\rm bounded\ for}\ t\geq0\},
\]
and ${\cal J}^{u}(\g)$ to be the space of \textit{unstable (orthogonal)
Jacobi fields} as 
\[
{\cal J}^{u}(\g)=\{J(t)\in{\cal J}(v):\ \left\Vert J(t)\right\Vert \ {\rm is}\ {\rm bounded\ for}\ t\leq0\}.
\]
Using these two linear spaces of ${\cal J}(\g)$ and the identification,
we define two subbundles $E^{s}(v)$ and $E^{u}(v)$ of $T_{v}T^{1}M$
as the following:
\begin{align*}
E^{s}(v):= & \{\xi\in T_{v}T^{1}M:\ J_{\xi}\in{\cal J}^{s}(v)\},\\
E^{u}(v):= & \{\xi\in T_{v}T^{1}M:\ J_{\xi}\in{\cal J}^{u}(v)\}.
\end{align*}

Last, we define $E^{c}(v)$ to be given by the flow direction.

\begin{defn}[Rank]\label{defn: rank} The \textit{rank} of $v \in T^1M$ is the dimension of the space of parallel Jacobi fields along $\gamma_v$.
We call $M$ a \textit{rank
1 manifold} if it has at least one rank 1 vector.
\end{defn}

\begin{defn} The \textit{singular set} is defined by $${\rm Sing}:=\{v\in T^{1}M:\ E^{s}(v)\cap E^{u}(v)\neq\emptyset\}.$$ 
The \textit{regular set} is defined as the complement of the singular set:
$${\rm Reg}:=T^{1}M\backslash{\rm Sing}.$$
\end{defn}

The following proposition summarizes known  facts regarding manifolds with no focal points.

\begin{prop}
\label{prop:no_focal_points} Let $M$ be a closed Riemannian manifold
without focal points. Then we have: 
\begin{enumerate}[font=\normalfont] 

\item \cite[\textcolor{black}{Theorem 3.2}]{Hurley:1986km} The geodesic flow ${\cal F}$ is topologically
transitive \textcolor{black}{if $M$ is rank 1}.  \footnotemark  \footnotetext[1]{\textcolor{black}{Ergodicity was claimed in \cite{Hurley:1986km} but the argument has an error. Nevertheless the proof for Theorem 3.2 is independent of ergodicity, and it remains valid.}}

\item\cite[Proposition 4.7, 6.2]{Pesin:1977vi} $\dim E^{u}(v)=\dim E^{s}(v)=n-1$,
and $\dim E^{c}(v)=1$ where $\dim M=n$. 

\item\cite[Theorem 4.11, 6.4]{Pesin:1977vi} The subbundles $E^{u}(v)$,
$E^{s}(v)$, $E^{cu}(v)$ and $E^{cs}(v)$ are ${\cal F}$--invariant
where $E^{cs}(v)=E^{c}(v)\oplus E^{s}(v)$ and $E^{cu}(v)=E^{c}(v)\oplus E^{u}(v)$.

\item\cite[Theorem 6.1, 6.4]{Pesin:1977vi} The subbundles $E^{u}(v)$,
$E^{s}(v)$, $E^{cu}(v)$ and $E^{cs}(v)$ are integrable to ${\cal F}$--invariant
foliations $W^{u}(v)$, $W^{s}(v)$, $W^{cu}(v)$ and $W^{cs}(v)$,
respectively. Moreover, $W^{u}(v)$ (resp. $W^{s}(v)$) consists of
vectors perpendicular to $H^{u}(v)$ (resp. $H^{s}(v)$) and toward
to the same side as $v$ (see below for the definition of the horospheres
$H^{s/u}(v)$). 

\item\cite[Lemma, p. 246]{Eschenburg:1977kn} $E^{u}(v)\cap E^{s}(v)\neq\emptyset$
if and only if $v\in{\rm Sing}$.

\item \cite[Theorem 1]{OSullivan:1976cf}, \cite[Theorem 2]{Eschenburg:1977kn}
\label{prop:the flat strip theorem}The Flat Strip Theorem: suppose
$M$ is simply connected and geodesics $\g_{1},\g_{2}$ are bi-asymptotic
in the sense that $d(\g_{1}(t),\g_{2}(t))$ is uniformly bounded for
all $t\in\R$. Then $\g_{1}$ and $\g_{2}$ bound a strip of flat
totally geodesically immersed surface.
\textcolor{black}{\item \cite[Section 5]{Eschenburg:1977kn} For any $J\in{\cal J}^{s}(\g)$(resp. ${\cal J}^{u}(\g)$), $||J(t)||$ is monotonely decreasing (resp. increasing) for all $t\in\mathbb{R}$.}
\end{enumerate}
\end{prop}

We shall introduce more metrics on $T^{1}M$ and the flow invariant
foliations introduced in Proposition \ref{prop:no_focal_points} in order to perform finer analysis. We write $d_{{\rm S}}$ for the distance
function on $T^{1}M$ induced by the Sasaki metric on $TT^{1}M$. \textcolor{black}{We will make use of another handy metric $d_{K}$ on $T^{1}M$:} 
\[
d_{K}(v,w):=\max\{d(\g_{v}(t),\g_{w}(t)):\ t\in[0,1]\}.
\]
Such metric $d_K$ also appeared in \cite{knieper1998uniqueness}.
It is not hard to see that $d_{{\rm S}}$ and $d_{K}$ are uniformly equivalent.
Thus, we will primarily work with \textcolor{black}{the metric} $d_{K}$ throughout
the paper. 
 
Furthermore, an \textit{intrinsic metric} $d^{s}$ on $W^{s}(v)$
for all $v\in T^{1}M$ is given by 
\[
d^{s}(u,w):=\inf\{l(\pi\g):\ \g:[0,1]\to W^{s}(v),\ \g(0)=u,\ \g(1)=w\}
\]
 where $l$ is the length of the curve in $M$, and the infimum is
taken over all $C^{1}$ curves $\g$ connecting $u,w\in W^{s}(v)$.
Using $d^{s}$ we define the \textit{local stable leaf} through
$v$ of size $\rho$ as: 
\[
W_{\rho}^{s}(v):=\{w\in W^{s}(v):\ d^{s}(v,w)\leq\rho\}.
\] 
Moreover, we locally define a similar intrinsic metric $d^{cs}$
on $W^{cs}(v)$ as: 
\[
d^{cs}(u,w)=|t|+d^{s}(f_{t}u,w)
\]
where $t$ is the unique time such that $f_{t}u\in W^{s}(w)$. This
metric $d^{cs}$ extends to the whole central stable leaf $W^{cs}(v)$.
We also define $d^{u}$, $d^{cu}$, $W_{\rho}^{u}(v),$ analogously. Notice
that when $\rho$ is small, these intrinsic metrics are uniformly equivalent
to $d_{{\rm S}}$ and $d_{K}$. 

\begin{rem}\label{rem: property s, cs}
A handy feature of these metrics is
that for any $v\in T^{1}M$, $\sigma\in\{s,cs\}$ and for any $u\in W^{\sigma}(w)$,
the map $t\mapsto d^{\sigma}(f_{t}u,f_{t}w)$ is a non-increasing function.
Indeed, let $\gamma$ be a curve in $W^s(v)$ connecting $u$ and $w$. Then $f_t\gamma$ lies in $W^s(f_tv)$. $\{f_s(\gamma)\}_{0\leq s\leq t}$ is a one-parameter family of geodesics and the associated Jacobi fields are all stable. Since stable Jacobi fields are non-increasing on manifolds without focal points (Proposition \ref{prop:no_focal_points} (7)), the length of $\gamma$ is not shorter than the length of $f_t(\gamma)$.

Similarly, for $\sigma\in\{u,cu\}$, the map $t\mapsto d^{\sigma}(f_{t}u,f_{t}w)$
is non-decreasing. These features are used in establishing the specification property in Section \ref{sec: proof of thm A}.
\end{rem}

Following Proposition \ref{prop:no_focal_points}, one can define
the \textit{stable horosphere} $H^{s}(v)\subset M$ and the \textit{unstable
horosphere} $H^{u}(v)\subset M$ as the projection of the respective
foliations to $M$:
\[
H^{s}(v)=\pi(W^{s}(v))\mbox{\ and\ } H^{u}(v)=\pi(W^{u}(v)).
\]

We now summarize some useful properties of the horospheres.

\begin{prop}\label{prop: horosphere}
\cite[Theorem 1 (i) (ii)]{Eschenburg:1977kn}\label{prop:horoshpere}
Let $M$ be a closed Riemannian manifold without focal points. Then
we have \begin{enumerate}[font=\normalfont] 

\item$H^{u}(v)$, $H^{s}(v)$ are $C^{2}$-embedded hypersurfaces
when lifted to the universal cover $\widetilde{M}$. 

\item For $\sigma\in\{s,u\}$, the symmetric linear operator $U^{\sigma}(v):T_{\pi v}H^{\sigma}(v)\to T_{\pi v}H^{\sigma}(v)$
given by $v\mapsto\nabla_{v}N$, i.e., the shape operator on $H^{\sigma}(v)$,
is well-defined, where $N$ is the unit normal vector field on $H^{\sigma}(v)$
toward the same side as $v$. 

\item$U^{u}$ is positive semi-definite and $U^{s}$
is negative semi-definite.
\end{enumerate}
\end{prop}

\subsection{Measure of hyperbolicity on $T^1M$}

For any $v\in T^1M$, let $\gamma_v$ be the unit speed geodesic starting with $v$. We introduce the following definition which serves as a measure of hyperbolicity on $T^1M$ throughout the paper.

\begin{defn}
For any $w\in T^1M$ with the same base point as $v$ and $w\perp v$, we denote by $J^u_w(t)$ (resp. $J^s_w(t)$) the unstable (resp. stable) Jacobi field along $\gamma_v$ with $J^u_w(0)=J^s_w(0)=w$. 
For any $T>0$, we define 
$$\lambda^u_T(v):=\min_{w \colon w\perp v}\left(\log \frac{||J^u_w(T)||}{||J^u_w(-T)||}\right) 
~~~
\text{ and } 
~~~
\lambda^s_T(v):=\min_{w \colon w\perp v}\left( -\log \frac{||J^s_w(T)||}{||J^s_w(-T)||}\right),$$
and
\begin{equation}\label{eq: lt}
\lambda_T(v):=\min\{\lambda^u_T(v), \lambda^s_T(v)\}.
\end{equation}
\end{defn}

Note that for any $v \in \Sing$, there exists an orthogonal Jacobi field $J$ along $\gamma_v$ such that $t \mapsto \|J(t)\|$ is a constant function. Hence, it follows that $\lt(v) = 0$ for any $v \in \Sing$ and $T>0$.
Using $\lt$, we set 
\begin{equation}\label{eq: Reg_T}
\Reg_T(\eta):=\{v\in T^1M \colon \lt(v) \geq \eta\}.
\end{equation}
Since $\lt|_\Sing \equiv 0$ for any $T>0$, $\Reg_T(\eta)$ is a compact subset of $\Reg$ for any $T,\eta>0$.

Throughout the paper, we will mostly be using $\lt$ as the measure of hyperbolicity on $T^1M$ and $\Reg_T(\eta)$ as the uniformity regular set.

However, we introduce another function $\lambda$ on $T^1M$ of similar nature. Let $\lambda^{u}(v)$ be the smallest eigenvalue of the second fundamental form $U^u(v)$, and $\lambda^s(v)$ be equal to $\lambda^u(-v)$. Define
\begin{equation}\label{eq: lambda}
\lambda(v) := \min(\lambda^u(v),\lambda^s(v)).
\end{equation}
Similar to $\Reg_T(\eta)$, we define $\Reg(\eta)$ using $\lambda$:
$$\Reg(\eta):=\{v\in T^1M \colon \lambda(v) \geq \eta\}.$$
We will make use of $\lambda$ and $\Reg(\eta)$ in Proposition \ref{prop: 8.1 of bcft}.

\begin{rem}\label{rem: lt vs lambda 1}
In \cite{burns2018unique} where results in this paper are proved for geodesic flows over nonpositively curved manifolds, $\lambda$ is used as the measure of hyperbolicity on $T^1M$. In their setting, the function $t \mapsto \|J(t)\|$ is convex for any Jacobi field $J$, and such convexity are used to deduce further estimates on $\lambda$. For instance, pointwise information such as $\lambda(v)=0$ can be used to extract information on the asymptotic behavior of the geodesic $\gamma_v$, and one can use this property to characterize the singular set via $\lambda$; see Lemma 3.2 and 3.3 of \cite{burns2018unique}.

For manifolds without focal points, we no longer have the convexity of the function $t \mapsto \|J(t)\|$. Although strictly weaker, we can still make use of the fact that the metrics $d^{s/u}$ are monotonic; see Remark \ref{rem: property s, cs}. So we need an alternative way to characterize the singular set suited to such monotonicity. This is done via $\lt$; see Lemma \ref{charsing}. 

In \cite{chen2018unique} where the results in \cite{burns2018unique} are extended to surfaces without focal points, the integral of $\lambda(f_sv)$ over $s \in [-T,T]$ served the similar purpose; there, such function was named $\lt$ and played the role of $\lambda$ from \cite{burns2018unique}. 

Our definition of $\lt$ defined in \eqref{eq: lt} is related to $\lt$ from \cite{chen2018unique} in that they agree when $M$ is a surface. However, in higher dimension, $\lt(v)$ as in \eqref{eq: lt} differs from the integral of $\lambda(f_sv)$ over $s \in [-T,T]$; other than the former being bigger than or equal to the latter, there is no direct relationship between the two.
\end{rem}

We now establish the properties of $\lt$ by studying the growth of Jacobi tensors. We survey the relevant results on Jacobi tensors below; see \cite{eschenburg1976growth} for details. 

Let $D^v$ be the stable Jacobi tensor along $\gamma_v$ (we may denote it by $D$ when there is no confusion). Namely, 
$$D:=\lim\limits_{s\to\infty}D_s,$$ where $D_s$ the Jacobi tensor along $\gamma_v$ uniquely determined by $D_s(0)=Id$ and $D_s(s)=0$. 
Any stable Jacobi field along $\gamma_v$ can be represented as 
\begin{equation}\label{eq: J tensor 1}
J^s_w(t)=D(t)w.
\end{equation}
The second fundamental form $U^s$ of the stable horosphere is given by 
\begin{equation}\label{eq: J tensor 2}
U^s(f_tv)=D'(t)D(t)^{-1}.
\end{equation}
Note that $U^s$ is symmetric and negative semi-definite. By simple computations, we know that $U^s$ satisfies the \textit{Riccati equation}:
$$U'+U^2+R=0.$$

\begin{lem}\label{Jgrowthlem}
For any $a<b$, and $w\perp v$, we have $$||J^s_w(b)||=||J^s_w(a)||\exp\left(\int_a^b \langle U^s(f_tv)(Y(t)), Y(t)\rangle dt\right),$$
where $Y(t):=D(t)w/||D(t)w||$.
\end{lem}
\begin{proof} The claim follows from \eqref{eq: J tensor 1}, \eqref{eq: J tensor 2}, and the definition of $Y(t)$: 
\begin{align*}
\log \frac{||J^s_w(b)||}{||J^s_w(a)||}&=\frac{1}{2}\int_a^b (\log ||J^s_w(t)||^2)'dt=\frac{1}{2}\int_a^b \frac{(||D(t)w||^2)'}{||D(t)w||^2}dt= \int_a^b \frac{\langle D'(t)w, D(t)w \rangle}{||D(t)w||^2}dt\\
&=\int_a^b \langle D'(t)D(t)^{-1}(Y(t)), Y(t)\rangle dt= \int_a^b \langle U^s(f_tv)(Y(t)), Y(t)\rangle dt.
\end{align*}
This completes the proof.
\end{proof}
With Lemma \ref{Jgrowthlem}, the following corollary is immediate from the definition of $\lt^s$:
\begin{cor}\label{deflem}
$\displaystyle\lambda_T^s(v)=\min\limits_{w\perp v}\left( \int\limits_{-T}^T \langle -U^s(f_tv)(Y(t)), Y(t)\rangle dt\right)$, where $Y(t)=D(t)w/||D(t)w||$.
\end{cor}
From negative semi-definiteness of $U^s$, we also obtain the following property:
\begin{cor}\label{incrsl}
$\la^s_T$ is non-decreasing with respect to $T$. Similarly, $\la^u_T$ and $\la_T$ are non-decreasing with respect to $T$.
\end{cor}

The following lemma characterizes the singular set using $\lt$.

\begin{lem}\label{charsing}
$\lambda_T (v)=0$ for all $T \in \mathbb{R}$ if and only if $v \in \Sing$.
\end{lem}
\begin{proof}
Without loss of generality, we may assume there exists $\{T_n\}_{n \in \N}$ such that $T_n\to \infty$ and $\lambda^s_{T_n}(v)=0$. By Corollary \ref{deflem} and positive semi-definiteness of $-U^s$, we can find $\{w_n\}_{n \in \N}$ satisfying $w_{n}\perp v$ such that,
$$U^s(f_tv)(D(t)w_n) = 0 \text{ for any }t\in [-T_n, T_n].$$

By passing to a subsequence if necessary, we may assume $w_n\to w$ for some $w \in T^1M$. Then $U^s(f_tv)(D(t)w)=0$ for all $t\in \R$. Thus $(J^s_w)'(t)\equiv 0$ and $J^s_{w}$ is a constant Jacobi field along $\gamma_v$. 
\end{proof}

\begin{rem}\label{rem: lt vs lambda 2}
As noted already in Remark \ref{rem: lt vs lambda 1}, $\lt$ defined in \eqref{eq: lt} is related but different from $\lt$ used in \cite{chen2018unique} for surfaces without focal points. Lemma \ref{charsing} is an example where such difference is manifested: $\lt$ used in \cite{chen2018unique}, defined as the integral of $\lambda$ along the orbit segment of length $2T$, may still be defined for manifolds without focal points of arbitrary dimension, but such $\lt$ does not necessarily satisfy Lemma \ref{charsing}. In fact, one motivations for introducing $\lt$ as in \eqref{eq: lt} is to establish the characterization of $\Sing$ using $\lt$ as done in Lemma \ref{charsing}. 
\end{rem}

We conclude this section by introducing and establishing the local product structure on compact subsets of $\Reg$.

\begin{defn}[Local product structure]\label{defn: local product}
Given $\kappa \geq 1$, $\delta>0$, and $v \in T^1M$, we say foliations $W^u$ and $W^{cs}$ have $(\kappa,\delta)$-\textit{local product structure at $v$} if for every $\ep \in (0,\delta]$ and any $w_1,w_2 \in B(v,\ep)$, the intersection of $W^u_{\kappa \ep}(w_1)$ and $ W^{cs}_{\kappa \ep}(w_2)$ consists of a single point denoted by $[w_1,w_2]$ which satisfies
$$\max \Big( d^u(w_1,[w_1,w_2]),d^{cs}(w_2,[w_1,w_2])\Big) \leq \kappa d_K(w_1,w_2).$$
\end{defn}

The following lemma establishes the local product structure of the foliations $W^u$ and $W^{cs}$ at every point of $v \in \Reg_T(\eta)$. Using the compactness of $\Reg_T(\eta)$ and the continuity of the distribution $E^u$ and $E^{cs}$, the proof of \cite[Lemma 4.4]{burns2018unique} works here without any modification.

\begin{lem}\cite[Lemma 4.4]{burns2018unique}\label{lem: local product structure}
For any $T,\eta>0$, there exist $\delta>0$ and $\kappa\geq 1$ such that the foliations $W^u$ and $W^{cs}$ has $(\kappa,\delta)$-local product structure at every $v \in \Reg_T(\eta)$.
\end{lem}

\section{Thermodynamic formalism}\label{sec: 3}

In this section, we describe a general theory of thermodynamic formalism; see \cite{walters2000introduction} for details. Throughout the section, let $(X,d)$ be a compact metric space, and let $\mathcal{F} = (f_t)_{t \in \R}$ be a continuous flow on $X$ and $\vphi \colon X \to \R$ be a potential (i.e. continuous function) on $X$. 

\subsection{Topological pressure}
\begin{defn}
For any $t,\delta>0$ and $x,y\in X$,

\begin{enumerate}[font=\normalfont] 
\item \textit{$d_t$-metric} is defined by $d_t(x,y):= \max\limits_{0 \leq \tau \leq t} d(f_\tau  x, f_\tau y)$.
\item The \textit{Bowen ball} of radius $\delta$ and order $t$
at $x$ is defined as 
\[
B_{t}(x,\delta)=\{y\in X:d_t(x,y) <\d\}.
\]

\item We say a set $E$ is $(t,\delta)-$\textit{separated} if $d_t(x,y) \geq \d$ for
all $x,y\in E$ with $x\neq y$.
\end{enumerate}
\end{defn}

\begin{defn}[Finite length orbit segments] \label{defn: finite length orbit}
Any subset $${\cal C}\subset X\times[0,\infty)$$
can be identified with a collection of \textit{finite length orbit
segments. More precisely, every $(x,t) \in {\cal C}$ is identified with the orbit segment $\{f_{\tau}x:\ 0\leq\tau\leq t\}.$}

We define $\displaystyle \Phi(x,t):=\int_{0}^{t}\vp(f_{\tau}x)d\tau$ to be the integral of $\vp$ along an orbit segment $(x,t)$.
\end{defn}

Let ${\cal C}_{t}:=\{x\in X:\ (x,t)\in{\cal C}\}$ be the set of length
$t$ orbit segments in ${\cal C}$. We define 
\[
\Lambda({\cal C},\vp,\delta,t)=\sup\Big\{\sum_{x\in E}e^{\Phi(x,t)}:\ E\subset{\cal C}_{t}\ {\rm is}\ (t,\delta)-{\rm separated}\Big\}.
\]

\begin{defn}
[Topological pressure] The \textit{pressure} of $\vp$ on ${\cal C}$
is defined as 
\[
P({\cal C},\vp)=\lim_{\delta\to0}\limsup_{t\to\infty}\frac{1}{t}\log\Lambda({\cal C},\vp,\delta,t).
\]
When ${\cal C=}X\times[0,\infty)$, we denote $P(X\times[0,\infty),\vp)$
by $P(\vp)$ and call it the \textit{topological pressure} of $\vp$
with respect to ${\cal F}$.
\end{defn}

Let ${\cal M}({\cal F})$ be the set of ${\cal F}$-invariant probability
measures on $X$. For any $\mu \in \mathcal{M}(\F)$, we set
$$P_\mu(\vphi):=h_\mu(\F)+\int \vphi d\mu.$$
The \textit{variational principle} states that $P(\vphi)$ is the supremum of $P_\mu(\vphi)$ over all $\mu \in \mathcal{M}(\F)$:
\begin{equation}\label{eq: var prin}
P(\vphi) = \sup\Big\{P_\mu(\vphi) \colon \mu \in \mathcal{M}(\F)\Big\}.
\end{equation}
Any $\mu \in \mathcal{M}(\F)$ that achieves the supremum is called an \textit{equilibrium state} of $\vphi$.

\begin{rem}
$\ $\begin{enumerate}[font=\normalfont] 

\item When the entropy map $\mu\mapsto h_{\mu}$ is upper semi-continuous,
any weak-$*$ limit of a sequence of invariant measures approximating the
pressure is an equilibrium state. In particular, there exists at least
one equilibrium state for every continuous potential.

\item In our setting of geodesic flows over manifolds without focal
points, the upper semi-continuity of the entropy map is guaranteed
by the entropy-expansivity established in \cite{liu2016entropy}.

\end{enumerate}
\end{rem}

\subsection{Gurevich pressure}

In this subsection, we will introduce the Gurevich pressure and the
equidistribution property, and we will use these in Section
\ref{sec: proof of thm B} to establish ergodic properties of the equilibrium states. Roughly speaking, the Gurevich pressure is given by the growth
rate of weighted regular closed geodesics. It is equal to the topological pressure
when the system is uniformly hyperbolic. In general, these two notions
of pressure are different; see \cite{Gelfert:2014hn} for more details.

Before going further, we begin by setting up the notations.
As before, $M$ denotes a Riemannian manifold, ${\cal F}=(f_{t})_{t\in\R}$ the geodesic flow on $T^{1}M$, and $\vp:T^{1}M\to\R$ a continuous
potential. We denote the set of regular closed geodesics with length
in $(a,b]$ by ${\rm Per}_{R}(a,b]$. For any closed geodesic $\g$,
we write
\[
\Phi(\g):=\int_{\g}\vp=\int_{0}^{|\g|}\vp(f_{t}v)dt
\]
 where $v\in T^{1}M$ is tangent to $\g$ and $|\g|$ is the length
of $\g$. Next, given $t,\Delta>0$, we define 
\[
\Lambda_{\Reg,\Delta}^{*}(\vp,t):=\sum_{\g\in{\rm Per}_{R}(t-\D,t]}e^{\Phi(\g)}.
\]

\begin{defn}
[Gurevich pressure] Given $\Delta>0$,

\begin{enumerate}[font=\normalfont]

\item The \textit{upper regular} \textit{Gurevich pressure} $\overline{P}_{\Reg,\D}^{*}$
of $\vp$ is defined as 

\[
\overline{P}_{\Reg}^{*}(\vp):=\limsup_{t\to\infty}\frac{1}{t}\log\Lambda_{\Reg,\Delta}^{*}(\vp,t).
\]
(Note that it is not hard to verify that $\overline{P}_{\Reg}^{*}(\vp)$ does not depend on $\Delta$.)
\item The\textit{ lower} \textit{regular Gurevich pressure} $\underline{P}_{\Reg,\D}^{*}$
of $\vp$ is defined as 
\[
\underline{P}_{\Reg,\D}^{*}(\vp):=\liminf_{t\to\infty}\frac{1}{t}\log\Lambda_{\Reg,\Delta}^{*}(\vp,t).
\]
If $\overline{P}_{{\rm Reg}}^{*}(\vp)=\underline{P}_{\Reg,\D}^{*}(\vp)$,
then we call this value the \textit{regular Gurevich pressure }and
denote it by $P_{\Reg}^{*}(\vp)$.

\end{enumerate}
\end{defn}

In what follows, we provide the precise definition of what it means for a measure to be
equidistributed along weighted regular closed geodesics.
\begin{defn}[Equidistribution]\label{defn: equidistribution} For a potential $\vp:T^{1}M\to\R$, we say $\mu$
is the \textit{weak-$*$ limit of $\vp-$weighted regular closed geodesics},
or \textit{equidistributed along weighted regular closed geodesics},
if for \textit{every} $\D>0$ we have
\begin{equation}\label{eq: equidistribution}
\mu=\lim_{t\to\infty}\frac{\sum_{\g\in{\rm Per}_{R}(t-\D,t]}e^{\Phi(\g)}\delta_{\g}}{\Lambda_{\Reg,\Delta}^{*}(\vp,t)}
\end{equation}
where $\delta_{\g}$ is the normalized Lebesgue measure along the
closed geodesic $\g$.
\end{defn}

\begin{rem}
We point out that in \cite{chen2018unique} there
also is a notion of a measure $\mu$ being the \textit{weak-$*$ limit
of $\vp-$weighted regular closed geodesics.} However, the notion defined here in Definition \ref{defn: equidistribution} is stronger. In fact, the notion in \cite{chen2018unique} only requires that there exists some $\Delta>0$ such that \eqref{eq: equidistribution} holds, whereas in the current paper it requires that \eqref{eq: equidistribution} holds independent of $\Delta$. The latter definition was first proposed by Parry \cite{parry1988equilibrium} for Axiom A flows.
\end{rem}

The next proposition can be found in \cite[Theorem 9.10]{walters2000introduction},
which states that equilibrium states can be constructed via weighted regular closed
geodesics.

\begin{prop}
\cite[Theorem 9.10]{walters2000introduction} \label{prop:walters_eq_states}Given
$\D>0$, suppose there exists $\{t_{k}\}_{k \in \N}$ with $t_k \to \infty$ such that 
\[
\lim_{k\to\infty}\frac{1}{t_{k}}\log {\Lambda_{\Reg,\Delta}^{*}(\vp,t_{k})=P(\vp)}
\]
and 
\[
\lim_{k\to\infty}\frac{\sum_{\g\in{\rm Per}_{R}(t_{k}-\D,t_{k}]}e^{\Phi(\g)}\delta_{\g}}{\Lambda_{\Reg,\Delta}^{*}(\vp,t_{k})}=\mu,
\]
 then $\mu$ is an equilibrium state.
\end{prop}

Therefore, we summary results above in the following proposition.

\begin{prop}
\label{prop:pressure_equal_equidistribution}If $\vp$ has a unique
equilibrium $\mu_{\vp}$ and $P_{\Reg}^{*}(\vp)=P(\vp)$, then $\mu_{\vp}$ is a weak-$*$ limit of $\vp-$weighted
regular closed geodesics.
\end{prop}

\subsection{Criterion for the unique equilibrium state}
In this subsection, we describe a set of criteria recently developed by Climenhaga and Thompson \cite{climenhaga2016unique} that establishes the existence of the unique equilibrium state. More specifically, \cite{climenhaga2016unique} extended Bowen's work on uniformly hyperbolic systems \cite{bowen1974some} as well as Franco's work on for hyperbolic flows \cite{Franco:1977jy} and developed based on it a set of criteria that has been successfully apply to many non-uniformly hyperbolic systems; see \cite{Climenhaga:2015wf}, \cite{burns2018unique}, \cite{chen2018unique}, and \cite{climenhaga2019equilibrium}.
First, we introduce various properties on the system $(X,\F)$ and the potential $\vphi \colon X \to \R $ necessary to state the Climenhaga-Thompson criteria.

\begin{defn}[Specification]\label{defn: specification} We say $\mathcal{C} \subset X \times [0,\infty)$ has \textit{specification} at scale $\rho>0$ if there exists $\T = \T(\rho)>0$ such that for every finite collection of elements in $\mathcal{C}$, i.e., $(x_1,t_1),\ldots,(x_n,t_n) \in \mathcal{C}$, and every $T_1,\ldots,T_n \in \R$ satisfying $T_1=0$ and $T_{i+1} - T_i \geq t_i +\T$ for each $1 \leq i  \leq n-1$, there exists $w \in T^1M$ such that
$$f_{T_i}w \in B_{t_i}(v_i,\rho) \text{ for each }1 \leq i \leq n.$$

We say $\mathcal{C}$ has \textit{specification} if it has specification at all scales. If $\mathcal{C} = X\times [0,\infty)$ has specification, then we say the flow has specification.
\end{defn}

\begin{defn}[Bowen property] We say a potential $\vphi \colon X \to \R$ has \textit{Bowen property on} $\mathcal{C} \subset X \times [0,\infty)$ if there exist $\ep,K>0$ such that for any $(x,t) \in \mathcal{C}$, we have
$$\sup\limits_{y \in B_t(x,\ep)}|\Phi(x,t) - \Phi(y,t)|\leq K$$
where $\displaystyle \Phi(x,t) := \int\limits_0^t \vphi(f_\tau x)d\tau$.
\end{defn}

\begin{defn}
[Decomposition of orbit segments] A \textit{decomposition} of $X\times[0,\infty)$
consists of three collections ${\cal P}$, ${\cal G}$, ${\cal S}\subset X\times[0,\infty)$
such that:\begin{enumerate}[font=\normalfont] 

\item There exist $p,g,s:X\times[0,\infty)\to\R$ such that for each
$(x,t)\in X\times[0,\infty)$, we have $t=p(x,t)+g(x,t)+s(x,t)$,
\item $(x,p(x,t))\in{\cal P}$, $(f_{p(x,t)}x,g(x,t))\in{\cal G}$,
and $(f_{p(x,t)+g(x,t)}x,s(x,t))\in{\cal S}$.

\end{enumerate}
\end{defn}

\begin{rem}\label{rem: lambda decomp}
Given any (lower semi-)continuous function $\lambda \colon X \to [0,\infty)$ and $\eta>0$, we define $B(\eta),G(\eta) \subset X \times [0,\infty)$ as follows:
\begin{equation}\label{eq: B(eta)}
B(\eta):= \Big\{ (x,t) \colon \frac{1}{t} \int_0^t \lambda(f_sx)ds \leq \eta\Big\}
\end{equation}
and 
\begin{equation}\label{eq: G(eta)}
G(\eta):= \Big\{ (x,t) \colon \frac{1}{\tau} \int_0^\tau \lambda(f_sx)ds \geq \eta \text{ and }\frac{1}{\tau}\int^\tau_{0}\lambda(f_{t-s}x)ds \geq \eta \text{ for all } s\in [0,t]\Big\}.
\end{equation}
Definition 3.4 from Call and Thompson \cite{call2019equilibrium} ensures that 
$(\mathcal{P},\mathcal{G}, \mathcal{S}):=(B(\eta),G(\eta),B(\eta))$ defines a decomposition of $X\times [0,\infty)$ with $p(x,t)$ being the largest value such that the initial segment upto time $p(x,t)$ is in $B(\eta)$, $s(x,t)$ being the largest value such that the terminal segment upto time $s(x,t)$ is in $B(\eta)$, and $g(x,t):=t-p(x,t)-s(x,t)$. Such method of decomposing the orbit segments using a (lower semi-)continuous function $\lambda$ is called the \textit{$\lambda$-decomposition} in \cite{call2019equilibrium}.
\end{rem}

Due to some technical reasons (see \cite{climenhaga2016unique}),
we need to work with collections slightly bigger than ${\cal P}$ and
${\cal S}$, namely,  
\begin{align*}
[{\cal P}] & :=\{(x,n)\in X\times\N:\ (f_{-s}x,n+s+t)\in{\cal P}\ {\rm for}\ {\rm some}\ s,t\in[0,1]\},
\end{align*}
and similarly for $[{\cal S}].$ The following definition is the last remaining piece needed to state the Climenhaga-Thompson criteria.

\begin{defn}
For any $x\in X$ and $\ep>0$,  
\begin{enumerate}[font=\normalfont] 

\item The \textit{bi-infinite Bowen ball} $\G_{\ep}(x)$ is defined
as 
\[
\G_{\ep}(x):=\{y\in X:\ d(f_{t}x,f_{t}y)\leq\ep\ {\rm for}\ {\rm all}\ t\in\R\}.
\]

\item The set of \textit{non-expansive points at scale} $\ep$ is
defined as
\[
{\rm NE(\ep}):=\{x\in X:\ \G_{\ep}\nsubseteq f_{[-s,s]}(x)\ {\rm for\ }{\rm any}\ s>0\}
\]
where $f_{[a,b]}(x)=\{f_{t}x:\ t\in[a,b]\}.$

\item \textit{The pressure of obstruction to expansivity} for $\vp$
is defined as 
\[
P_{{\rm exp}}^{\perp}(\vp):=\lim_{\ep\to0}P_{\exp}^{\perp}(\vp,\ep)
\]
 where 
\[
P_{\exp}^{\perp}(\vp,\ep):=\sup\{h_{\mu}(f_{1})+\int\vp d\mu:\ \mu\in{\cal M}^{e}({\cal F})\ {\rm and\ }\mu({\rm NE}(\ep))=1\}
\]
 and ${\cal M}^{e}({\cal F})$ is the set of ${\cal F}-$invariant
ergodic probability measures on $X$. 

\end{enumerate}
\end{defn}

\begin{rem}
For uniform hyperbolic systems, ${\rm NE}(\ep)=\emptyset$ for all $\ep$
sufficiently small; thus $P_{\exp}^{\perp}(\vp)=-\infty$. Hence, the pressure of obstruction to expansivity $P_{{\rm exp}}^{\perp}(\vp)$ is always necessarily smaller than the pressure $P(\vp)$ in the setting of Bowen's work \cite{bowen1974some}.
\end{rem}

Finally, the following theorem describes the Climenhaga-Thompson criteria for the uniqueness of the equilibrium states. We will use this theorem to prove Theorem \ref{thm: 1} in Section \ref{sec: proof of thm A}.

\begin{thm}
\cite[Theorem A]{climenhaga2016unique}\label{thm:C-T_criteria} Let
$(X,{\mathcal F})$ be a flow on a compact metric space, and $\vphi:X\to\R$
be a continuous potential. Suppose that $P_{{\rm exp}}^{\perp}(\vphi)<P(\vphi)$
and $X\times[0,\infty)$ admits a decomposition $({\mathcal P},{\mathcal G},{\mathcal S})$
with the following properties: 
\begin{enumerate}[label=(\roman*)]
\item ${\mathcal G}$ has specification;
\item $\vphi$ has Bowen property on ${\mathcal G}$;
\item\label{eq: CT 3} $P([{\mathcal P}]\cup[{\mathcal S}],\vphi)<P(\vphi)$.
\end{enumerate}
Then $(X,{\mathcal F},\vphi)$ has a unique equilibrium state $\mu_{\vphi}$.
\end{thm}

\section{Proof of Theorem \ref{thm: 1}}\label{sec: proof of thm A}
Let $\vphi$ be a \hol potential on $T^1M$ such that $P(\vphi)>P(\Sing,\vphi):=P(\Sing \times [0,\infty),\vphi)$.
We will show that $\vphi$ has a unique equilibrium state $\mu_\vphi$ by choosing a suitable decomposition of $T^1M \times [0,\infty)$ and verifying each condition of Theorem \ref{thm:C-T_criteria}. Throughout the section, the metric $d$ on $T^1M$ refers to the metric $d_K$.

Recall that $U^u(v)$ is the second fundamental form introduced in Proposition \ref{prop: horosphere}. Let $\Lambda$ be the maximum eigenvalue of $U^u(v)$ over all $v \in T^1M$.
Then the definition of $d_K$, $d_t$, and Remark \ref{rem: property s, cs} imply that various metrics are related as follows:
\begin{align}
\begin{split}d_{t}(v,w) & \leq d^{cs}(v,w),\\
d_{t}(v,w) & \leq d^{u}(f_{t+1}v,f_{t+1}w)\leq e^{\Lambda}d^{u}(f_{t}v,f_{t}w).
\end{split}
\label{eq: metric comparison}
\end{align}
Such relationships between various metrics will be used in establishing the specification property later in this section.

\subsection{Decomposition $(\mathcal{B}_T(\eta),\mathcal{G}_T(\eta),\mathcal{B}_T(\eta))$}\label{subsec: decomp}

For any $T,\eta>0$, we define $\mathcal{B}_T(\eta)$ and $\mathcal{G}_T(\eta)$ by \eqref{eq: B(eta)} and \eqref{eq: G(eta)} using the function $\lt \colon T^1M \to [0,\infty)$ defined as in \eqref{eq: lt}:
$$B_T(\eta):= \Big\{ (x,t) \colon \frac{1}{t} \int_0^t \lt(f_sx)ds \leq \eta\Big\}$$
and
$$G_T(\eta):= \Big\{ (x,t) \colon \frac{1}{\tau} \int_0^\tau \lt(f_sx)ds \geq \eta \text{ and }\frac{1}{\tau}\int^\tau_{0}\lt(f_{t-s}x)ds \geq \eta \text{ for all } s\in [0,t]\Big\}.
$$

Since $\lt$ is a continuous function for any $T,\eta>0$, Remark \ref{rem: lambda decomp} ensures that 
\begin{equation}\label{eq: decomp by lt}
(\mathcal{P},\mathcal{G},\mathcal{S})=(\mathcal{B}_T(\eta),\mathcal{G}_T(\eta),\mathcal{B}_T(\eta))
\end{equation}
defines a decomposition of $T^1M\times [0,\infty)$.
The following lemma shows that the geodesic flow $\mathcal{F}=\{f_t\}_{t\in\R}$ displays hyperbolic behaviors along the orbit segments in $\mathcal{G}_T(\eta)$.

\begin{lem}\label{lem45}
For any $T,\eta>0$, there exists $C=C(T,\eta)$ and $\delta  =\delta(T,\eta)>0$ such that the followings hold: let $(v,t)$ be any orbit segment in $ \mathcal{G}_T(\eta)$, then  
\begin{enumerate}[label=(\alph*)]
\item\label{eq: a}
For any $f_tw,f_tw'\in W^{u}_{\delta}(f_tv)$, and any $0\leq\tau\leq t$, 
$$d^{u}(f_{\tau}w, f_{\tau}w')\leq Cd^{u}(f_tw,f_tw')e^{-\frac{(t-\tau)\eta}{4T}}.$$
\item\label{eq: b} Moreover, denoting by $(\wt{v},\wt{t}):=(f_{-T}v,t+2T)$ the orbit segment obtained by extending $(v,t)$ by time $T$ at both ends, for any $f_{\wt{t}}w,f_{\wt{t}}w' \in W^u_\delta(f_{\wt{t}}\wt{v})$, 
$$d^u(w,w') \leq d^u(f_{\wt{t}}w,f_{\wt{t}}w')e^{-\frac{t\eta}{4T}}.$$ 
\end{enumerate}
Analogous inequalities hold for $d^{s}$ and $d^{cs}$.
\end{lem}

The proof of Lemma \ref{lem45} requires a general lemma on double integrals.

\begin{lem}\cite[Lemma 3.12]{chen2018unique} \label{lem: double int}
For any non-negative integrable function $\psi \colon \R \to \R$, let $$\psi_T(x):= \int\limits_{-T}^T \psi(x+\tau)~d\tau.$$
Then for every $a\leq b$, 
$$\int_a^b \psi_T(t)dt \leq 2T\int_{a-T}^{b+T} \psi(t)dt.$$
\end{lem}


\begin{proof}[Proof of Lemma \ref{lem45}]The proof for \ref{eq: a} and \ref{eq: b} are similar.
Using uniform continuity of $\lt$, choose $\delta>0$ small such that 
$$d_K(v,w)<\delta \implies |\lt(v)- \lt(w)| \leq \eta/2.$$

For \ref{eq: a},
let $c:[0,1]\to W^u_{\delta}(f_tv)$ be a curve connecting $f_tw$ with $f_tw'$ whose length realizes the distance $d^u(f_tw,f_tw')$. Each $r \in [0,1]$ determines a geodesic $\gamma_r$ perpendicular to the unstable horosphere $\pi (v)$. Such one-parameter family of geodesics generate a family of unstable Jacobi fields $J_r$. Setting $Y_r(\tau):=J_r(\tau)/\|J_r(\tau)\|$, Lemma \ref{Jgrowthlem} and \ref{lem: double int} and Corollary \ref{deflem} give
\begin{align*}
\log \frac{\|J_r(t)\|}{\|J_r(\tau)\|} &= \int_{\tau}^t \langle U^u(f_{s-t}c(r))Y_r(s),Y_r(s)\rangle ds\\
&\geq \frac{1}{2T}\int_{\tau+T}^{t-T} \int_{-T}^T \langle U^u(f_{s+\tau-t}c(r))Y_r(s+\tau),Y_r(s+\tau)\rangle d\tau ds\\
&\geq \frac{1}{2T} \int_{\tau+T}^{t-T} \lt^u(f_{s-t}c(r))ds\\
&\geq \frac{1}{2T}\int_{\tau}^t \lt^u(f_{s-t}c(r))ds - 2T\Lambda\\
&\geq \frac{(t-\tau)\eta}{4T} -2T\Lambda.
\end{align*}
Here we used the fact that $U^u$ is positive semi-definite in order to apply Lemma \ref{lem: double int}. Setting $C = e^{2T\Lambda}$ gives \ref{eq: a}.

For \ref{eq: b}, we similarly define $c \colon [0,1] \to W^u_\delta(f_{\wt{t}}v)$ connecting $f_{\wt{t}}w$ to $f_{\wt{t}}w'$ and obtain a family of unstable Jacobi fields $J_r$. Using the same notations as part \ref{eq: a}, we have
$$\log \frac{\|J_r(\wt{t})\|}{\|J_r(0)\|} \geq \frac{1}{2T}\int_T^{\wt{t}-T}\lt^u(f_{s-\wt{t}}c(r))ds \geq \frac{t\eta}{4T}.$$
This completes the proof.
\end{proof}

Using Lemma \ref{lem45} and \hol continuity of $\vphi$, the Bowen property of $\vphi$ on $\mathcal{G}_T(\eta)$ easily follows:

\begin{prop}\cite[Theorem 6.5]{chen2018unique}
$\vphi$ has Bowen property on $\mathcal{G}_T(\eta)$ for any $T,\eta>0$.
\end{prop}
\begin{proof}
Using part \ref{eq: a} of Lemma \ref{lem45}, it follows that $\vphi$ has the Bowen property along the stable and unstable leaves of $\mathcal{G}_T(\eta)$, and together imply that $\vphi$ has the Bowen property on $\mathcal{G}_T(\eta)$.
See \cite{chen2018unique} for details.
\end{proof}

\subsection{Specification}
For the remainder of this section, we describe how the decomposition
$(\mathcal{B}_T(\eta),\mathcal{G}_T(\eta),\mathcal{B}_T(\eta))$ given in \eqref{eq: decomp by lt} satisfies the conditions in Theorem \ref{thm:C-T_criteria}. Since much of the following subsections follow the corresponding sections in \cite{burns2018unique} and \cite{chen2018unique} closely, we will only sketch the main ideas and refer the readers there for details.

Let $\CC_T(\eta)$ be the set of orbit segments whose endpoints belong to $\Reg_T(\eta)$:
$$\CC_T(\eta):= \{(v,t) \in T^1M \times [0,\infty) \colon v,f_tv \in \Reg_T(\eta)\}.$$
Since $\mathcal{G}_T(\eta)$ is a subset of $\CC_T(\eta)$, the following proposition shows that $(\mathcal{B}_T(\eta),\mathcal{G}_T(\eta),\mathcal{B}_T(\eta))$ meets the first condition of Theorem \ref{thm:C-T_criteria} for all $T,\eta>0$. 
\begin{prop}\cite[Theorem 4.1]{burns2018unique}\label{prop: specification}
For any $T,\eta>0,$ $\CC_T(\eta)$ has specification defined as in Definition \ref{defn: specification}.
\end{prop}

\begin{rem}
Given a compact subset $X \subset \Reg$, Lemma \ref{charsing} and compactness of $X$ give some $T,\eta>0$ such that $X \subset \Reg_T(\eta)$. In particular, this implies that Proposition \ref{prop: specification} also holds on a set of orbit segments whose endpoints lie in some compact subset $X\subset \Reg$. This version of Proposition \ref{prop: specification} will appear in Subsection \ref{subsec: proof of thm C} when applied to $X=\Reg(\eta)$ for some $\eta>0$. 
\end{rem}

The proof of Proposition \ref{prop: specification} needs a few lemmas concerning the properties of the foliations $W^{s/u}$.

\begin{lem}\label{minfol}
The foliations $W^{s/u}$ are minimal.
\end{lem}
\begin{proof}
The analogous result for nonpositively curved manifolds was proved in \cite[Theorem 3.7]{ballmann1982axial}. The idea is to first establish the density of $W^s(v)$ for some $v\in T^1M$, then carry out the argument in the proof of \cite[Theorem 5.5]{Eberlein1973II} 
without the visibility condition. For geodesic flows over manifolds without focal points, we also know that at least one of the stable leaves $W^s(v)$ is dense in $T^1M$ from \cite{Hurley:1986km}. Then, the argument in \cite{Eberlein1973II} follows without much modification when $M$ is compact and has no focal points.
\end{proof}
Using the compactness of $T^1M$, we have the following corollary:
\begin{cor}\cite[Lemma 8.1]{Climenhaga:2015wf}\label{cor: uniform minimality}
For every $\ep>0$, there exists $R>0$ such that $W^{s/u}_R(v)$ is $\ep$-dense in $T^1M$ for any $v\in T^1M$.
\end{cor}

\begin{lem}\label{lem: 2}
Given $\delta>0$, there exist $T,\eta>0$ such that if $\lambda_T(v) \leq \eta$, then $d_K(v,\Sing)<\delta$.
\end{lem}
\begin{proof}
From Lemma \ref{charsing}, $\Reg = \bigcup\limits_{\eta,T>0}A(\eta,T)$ where $A(\eta,T):=\{v \colon \lambda_T(v)>\eta\}$. Since $K := \{v \in T^1M \colon d_K(v,\Sing) \geq \delta\}$ is a compact set contained in $\Reg$, $\{A(\eta,T)\}_{\eta,T>0}$ is a nested open cover of $K$. So there exist $\eta,T>0$ such that $K \subset A(\eta,T)$
\end{proof}

\begin{rem}
Lemma \ref{lem: 2} serves the similar purpose as \cite[Proposition 3.4]{burns2018unique}, except that we need to use $\lt$ instead of $\lambda$ because the characterization of the singular set from Lemma \ref{charsing} is done using $\lt$; see also Remark \ref{rem: lt vs lambda 2}. 
\end{rem}

Next lemma follow as in \cite{burns2018unique} using the flat strip theorem \cite{OSullivan:1976cf} and Dini's theorem. 

\begin{lem} \cite[Proposition 3.13]{burns2018unique}\label{lem: 3}
For any $\ep,R>0$ and any compact $X \subset \Reg$, there exists $\T>0$ such that for every $w \in X$ and $t \geq \T$, we have $f_tW^u_\ep(w) \supset W^u_R(f_tw)$ and $f_{-t}W^s_{\ep}(w) \supset W^s_R(f_{-t}w)$.
 \end{lem}

The following lemma is also based on \cite{burns2018unique}. The only difference between the lemma stated below from its corresponding version in \cite{burns2018unique} is that $\Reg_T(\eta)$ is playing the role of $\Reg(\eta)$. However, this does not cause any issue in extending the proof because all the ingredients going into the proof are also available in our setting; such ingredients include Lemma \ref{lem: local product structure} for $\Reg_T(\eta)$ as well as Corollary \ref{cor: uniform minimality} and Lemma \ref{lem: 3}.

\begin{lem}\cite[Lemma 4.5]{burns2018unique}\label{lem: connect}
Given $T,\eta>0$, there exists $\delta>0$ such that for any $\rho \in (0,\delta]$ there exists $\T:=\T(\rho)>0$ such that for every $v,w \in \Reg_T(\eta)$ and $v' \in B(v,\delta)$ and $w'\in B(w,\delta)$, we have $f_t(W^u_\rho(v')) \cap W^{cs}_{\rho}(w') \neq \emptyset$ for every $t\geq \T$.
\end{lem} 
%

We are now ready to prove Proposition \ref{prop: specification}. While the statements and lemmas going into the proof are slightly different, the proof closely follows that of \cite[Proposition 5.5]{chen2018unique} with the role of \cite[Lemma 5.3]{chen2018unique} replaced by Lemma \ref{lem: connect}.
 
\begin{proof}[Proof of Proposition \ref{prop: specification}]
Let $T,\eta>0$ be given. We begin by fixing a regular closed geodesic $(v_0',t_0')$ as our reference orbit. From Lemma \ref{charsing}, there exist some $T',\eta'>0$ such that the entire orbit segment $(v_0',t_0')$ is contained in $\Reg_{T'}(\eta')$. By comparing $T'$ and the given $T$, we re-define $T$ as the larger of the two. Similarly, we re-define $\eta$ as the smaller of $\eta'$ and the given $\eta$. It then follows that $(v_0',t_0') \in \mathcal{G}_{T}(\eta)$. We set $v_{0}:=f_{-T}v_{0}'$ and
$t_{0}:=2T+t_{0}'$. Then $(v_{0},t_{0})$ is an extended orbit
segment obtained from $(v_{0}',t_{0}')$ whose endpoints $v_0,f_{t_0}v_0$ belong to $\Reg_{T}(\eta)$. 

From the uniform continuity of $\lt$, there exists $\delta_1>0$ such that
$$d_K(v,w) <\delta_1 \implies |\lt(v) - \lt(w)|<\eta/2.$$
By decreasing $\delta_1$, if necessary, we may suppose that $\delta_1$ is less than or equal to $\delta(T,\eta)$ from Lemma \ref{lem45}.
Then for any $w \in B_{t_0}(v_0,\delta_1)$, we have $(f_{T}w,t_0') \in \mathcal{G}_T(\eta/2)$. Setting $\alpha:=\exp\Big(-\frac{t_0'\eta}{8T}\Big)\in (0,1)$, it then follows from part \ref{eq: b} of Lemma \ref{lem45} that for any $w \in B_{t_0}(v_0,\delta_1)$ and $w' \in f_{-t_0}W^u_{\delta_1}(f_{t_0}w)$, we have
\begin{equation}\label{eq: decrease by alpha}
d^u(w,w')\leq \alpha d^u(f_{t_0}w,f_{t_0}w').
\end{equation}

Let $\delta_2>0$ be given by Lemma \ref{lem: connect}, and set $\delta:=\min\{\delta_1,\delta_2\}$.
Let $\rho \in (0,\delta)$ be an arbitrary given scale. Setting $$\rho':= \frac{1}{6}\rho e^{-\Lambda}(1-\alpha),$$ we will show that $\CC_T(\eta)$ has specification with $\T:=t_0+2\T_0$ where $\T_0:=\T(\rho')$ is from Lemma \ref{lem: connect}. 

Let $(v_1,t_1),\ldots,(v_n,t_n) \in\mathcal{G}_T(\eta)$ and $T_1=0,T_2,\ldots,T_n \in \R_{\geq 0}$ be given such that $T_{i+1}-T_i \geq t_i +\T$ for each $1 \leq i \leq n-1$. We will inductively build a sequence $w_1,\ldots,w_n \in T^1M$ with
\begin{equation}\label{eq: induction}
f_{T_j}w_j \in W^{cs}_{\rho'}(v_j) 
\end{equation}
such that $w:=w_n$ satisfies 
$$f_{T_i}w \in B_{t_i}(v_i,\rho) \text{ for each }1 \leq i \leq n.$$
We begin by setting $w_1=v_1$. Since $T_1=0$, we clearly have \eqref{eq: induction} for $j=1$. For the inductive step, we set up a few notations:
$$s_i := T_i+t_i\text{ and }\ell_i : = s_i+\T_0+t_0.$$
Suppose we have $w_j \in T^1M$ satisfying \eqref{eq: induction}. From $w_j\in T^1M$ we will build $w_{j+1} \in T^1M$ such that its forward orbit shadows that of $w_j$ for time $s_j$, and then after time $\T_0$ starts shadowing the orbit segment $(v_0,t_0)$, and finally satisfies \eqref{eq: induction} at time $T_{j+1}$.
To do so, consider $f_{s_j}w_j$ and $v_0$. Since $f_{T_j}w_j \in W^{cs}_{\rho'}(v_j)$ from the inductive hypothesis and $d^{cs}$ is non-increasing in forward time (Remark \ref{rem: property s, cs}), we have $f_{s_j}w_j \in W^{cs}_{\rho'}(f_{t_j}w_j)$. Lemma \ref{lem: connect} applied to $f_{s_j}w_j$ and $v_{0}$ implies that there exists $u_j \in T^1M$ such that
$$f_{s_j}u_j \in W^u_{\rho'}(f_{s_j}w_j) \text{ and } f_{s_j+\T_0}u_j\in W^{cs}_{\rho'}(v_0).$$
Using again the fact that $d^{cs}$ is non-increasing in forward time, we have $f_{\ell_j}u_j \in W^{cs}_{\rho'}(f_{t_0}v_0)$. Since $T_{j+1}-T_j \geq t_j +\T$ with $\T:=t_0+2\T_0$, we have $T_{j+1}-\ell_{j} \geq \T_0$ and applying Lemma \ref{lem: connect} to $f_{\ell_j}u_j$ and $v_{j+1}$ gives $w_{j+1} \in T^1M$ such that
$$f_{\ell_j}w_{j+1} \in W^u_{\rho'}(f_{\ell_j}u_j) \text{ and }f_{T_{j+1}}w_{j+1} \in W^{cs}_{\rho'}(v_{j+1}).$$
Since $d^u$ is non-increasing in backward time, we have
$$d^u(f_{s_j}w_{j+1},f_{s_j}w_j) \leq d^u(f_{\ell_j}w_{j+1},f_{\ell_j}u_j)+d^u(f_{s_j}u_{j},f_{s_j}w_j) \leq  \rho' +\rho' =2\rho'.$$
From \eqref{eq: decrease by alpha}, each time the backward orbits of $f_{s_j}w_{j+1}$ and $ f_{s_j}w_j$ pass nearby the orbit segment $(v_0,t_0)$ their $d^u$-distance decrease by a factor of $\alpha$.
Hence, for any $1 \leq i \leq j$,
$$d^u(f_{s_i}w_{j+1},f_{s_i}w_j) \leq 2\rho'\alpha^{j-i}.$$
It then follows that 
$$d^u(f_{s_i}w_{j+1},f_{s_i}w_i) \leq \sum\limits_{m=i}^j d^u(f_{s_i}w_{m+1},f_{s_i}w_m) \leq 2 \sum\limits_{m=i}^j\rho'\alpha^{j-m}\leq \frac{2\rho'}{1-\alpha}.$$
From \eqref{eq: metric comparison}, we have $$d_{t_i}(f_{T_i}w_{j+1},v_i) \leq d_{t_i}(f_{T_i}w_{j+1},f_{T_i}w_{i})+d_{t_i}(f_{T_i}w_i,v_i) \leq \frac{2\rho' e^\Lambda}{1-\alpha} + \rho' = \frac{1}{3}\rho+\rho' <\rho.$$
Hence, this proves that $w:=w_n$ satisfies $f_{T_i}w \in B_{t_i}(v_i,\rho) \text{ for each }1 \leq i \leq n$.
\end{proof}

From Proposition \ref{prop: specification}, we have the following version of the closing lemma whose formulation resembles \cite[Lemma 4.7]{burns2018unique}. In fact, its proof readily extends for the following lemma:

\begin{lem}\label{lem: closing} For all $\rho, \eta, \ep, T>0$, there exists $\T>0$ such that for every $(v,t) \in \CC_T(\eta)$, there exist $w\in B_t(v,\rho)$ and $\tau \in [\T-\ep,\T+\ep]$ such that $f_{t+\tau}w = w$.
\end{lem}
\begin{proof}
Instead of using any arbitrary $(v_0,t_0) \in \CC(\eta)$ as a reference orbit as in \cite[Lemma 4.7]{burns2018unique}, we instead proceed similarly to the proof of Proposition \ref{prop: specification} by fixing any regular closed geodesic $(v_0',t_0')$. After possibly re-defining $T,\eta$ as in the proof of Proposition \ref{prop: specification}, we make use of the extended orbit segment $(v_0,t_0):=(f_{-T}v_0',2T+t_0') \in\mathcal{G}_T(\eta)$. The rest of the proof proceed as in the proof of \cite[Lemma 4.7]{burns2018unique} using the Brouwer fixed point theorem.  
\end{proof}

The following corollary is the corresponding version of \cite[Corollary 4.8]{burns2018unique}. Using Lemma \ref{lem: closing} in place of \cite[Lemma 4.7]{burns2018unique} and $\lt$ replacing $\lambda$, the proof follows verbatim. 
\begin{cor}
\label{cor:strong_closing_lemma}
For all $\rho,\eta, \ep,T,M>0$, there exists $\T>0$ such that for every $(v,t)$ with the property that there exist $p,s\in [0,M]$ such that $(f_pv,f_{t-p-s}v)\in \CC_T(\eta)$, there exists $w \in B_t(v,\rho)$ and $\tau \in [\T-\ep,T+\ep]$ such that $f_{t+\tau}w = w$ and $w \in \Reg$.
\end{cor}

\subsection{Role of the pressure gap assumption: $P(\vphi)>P(\Sing,\vphi)$}

In this subsection, we describe how the pressure gap assumption on $\vphi$ can be used to meet two remaining assumptions in Theorem \ref{thm:C-T_criteria} using the geometry of manifolds without focal points. 

Given a set of orbit segments $\cal{C} \subset T^1M \times [0,\infty)$, $\cal{M}(\cal{C})$ refers to the set of all $\F$-invariant measures on $T^1M$ obtained as weak-$*$ limits of convex combinations of empirical measures along orbit segments in $\cal{C}$; see \cite[Section 5]{burns2018unique}. Moreover, we are denoting $\cal{M}(\Sing \times [0,\infty))$ by $\cal{M}(\Sing)$.

Setting $\displaystyle \mathcal{M}_{\lt}(\eta):=\Big\{\mu \in \mathcal{M}(\mathcal{F}) \colon \int \lt~d\mu \leq \eta\Big\}$, Lemma \ref{charsing} and the entropy-expansivity of the geodesic flows over manifolds with no focal points \cite{liu2016entropy} imply that
\begin{equation}\label{eq: M(Sing)}
\mathcal{M}(\Sing) =\bigcap\limits_{T,\eta>0} \mathcal{M}_{\lt}(\eta);
\end{equation}
see \cite[Proposition 5.2]{burns2018unique} and \cite[Proposition 7.3]{chen2018unique}.
Note that $\mathcal{M}_{\lt}(\eta)$ is a nested subset of $\mathcal{M}(\mathcal{F})$ that converges to $\mathcal{M}(\mathcal{\Sing})$ as $T \to \infty$ and $\eta \to 0$. It follows that the third item \ref{eq: CT 3} of Climenhaga-Thompson's criteria holds for all sufficiently large $T$ and sufficiently small $\eta$:
 
\begin{prop}\cite[Proposition 7.3]{chen2018unique}\label{prop: CKP7.3}
There exist $T_0,\eta_0>0$ such that $P(\vphi)>P([\mathcal{B}_{T}(\eta)],\vphi)$ for any $T \geq T_0$ and $\eta \leq \eta_0$.
\end{prop}
\begin{proof}
From \eqref{eq: M(Sing)}, given any $\ep>0$, there exist $T_0,\eta_0>0$ such that for any $\mu \in \mathcal{M}_{\la_{T_0}}(\eta_0)$, 
$$P_\mu(\vphi) <  P(\Sing,\vphi)+\ep.$$ 
Since $\mathcal{M}([\mathcal{B}_{T_0}(\eta_0)])$ is a subset of $\mathcal{M}_{\la_{T_0}}(\eta_0)$,
by choosing $\ep$ smaller than the size of the pressure gap $P(\vphi)- P(\Sing,\vphi)$, the corresponding $T_0,\eta_0$ are the required constants. See \cite[Proposition 7.3]{chen2018unique} for details.
\end{proof}

Using the similarity of geometric features between manifolds without focal points and nonpositively curved manifolds, we have the following estimate on $P^\perp_{\text{exp}}(\vphi)$.

\begin{prop}\cite[Proposition 5.4]{burns2018unique}
$P^\perp_{\text{exp}}(\vphi)\leq P(\Sing,\vphi)$
\end{prop}
\begin{proof}
The proof for \cite[Proposition 5.4]{burns2018unique} works here without any modifcation. In fact, the main argument of the proof there is that the flat strip theorem \cite{eberlein1996geometry} implies that $\text{NE}(\ep) \subset \Sing$ for all small enough $\ep>0$. Since the flat strip theorem remains to hold for manifolds with no focal points \cite{OSullivan:1976cf}, so does the proof. See \cite[Proposition 5.4]{burns2018unique} for details.
\end{proof}
With the pressure gap assumption on $\vphi$, we immediately obtain the following corollary and complete the proof of Theorem \ref{thm: 1}.

\begin{cor}
$P(\vphi)>P^\perp_{\text{exp}}(\vphi)$.
\end{cor}

\section{Proof of Theorem \ref{thm: 2}}\label{sec: proof of thm B}

\subsection{Equidistribution of the weighted regular closed geodesicser}

In this subsection, we show the unique equilibrium state is the weak-$*$
limit of the weighted regular closed geodesics. 
\begin{prop}
\label{prop:equi_distrib}The unique equilibrium state $\mu$ is the
weak-$*$ limit of weighted regular closed geodesics, that is, for all
$\Delta>0$
\[
\mu=\lim_{t\to\infty}\frac{\sum_{\gamma\in{\rm Per}_{R}(t-\D,t]}e^{\Psi(\g)}\delta_{\g}}{\Lambda_{\Reg,\D}^{*}(\vp,t)}
\]
where $\delta_{\gamma}$ is the normalized Lebegue measure supported
on the closed geodesic $\g.$
\end{prop}

The above proposition follows directly from the following lemma and
Proposition \ref{prop:pressure_equal_equidistribution}.
\begin{lem}
For any $\D>0$, there exists $\beta>0$ such that the regular closed
geodesics satisfy 
\[
\frac{\beta}{t}e^{tP(\vp)}\leq\Lambda_{\Reg,\D}^{*}(\vp,t)\leq\beta^{-1}e^{tP(\vp)}
\]
for $t$ big enough. 
\end{lem}

\begin{proof}
It follows the proof of \cite[Proposition 6.4]{burns2018unique} after
replacing \cite[Corollary 4.8]{burns2018unique} in the proof by Corollary
\ref{cor:strong_closing_lemma}.
\end{proof}
\begin{rem}
We point out again that Proposition \ref{prop:equi_distrib}
is stronger than the similar result found in \cite[Proposition 8.13]{chen2018unique}. Indeed, Proposition \ref{prop:equi_distrib} works for aribtrarily small $\Delta>0$, whereas
$\Delta$ had to be sufficient large in \cite[Proposition 8.13]{chen2018unique}. This is due to slightly different and stronger 
\end{rem}

\subsection{The $K$-property on $\mu_\vphi$}
In this subsection, we show that the unique equilibrium state $\mu_\vphi$ from Theorem \ref{thm: 1} has the $K$-property via methods developed by Call and Thompson \cite{call2019equilibrium} based on the earlier work of Ledrappier \cite{ledrappier1977mesures}. 

\begin{defn}[$K$-property]
Let $(X,\mathcal{B},\mu)$ be a probability space, and let $f$ be a measure-preserving invertible transformation of $(X,\mathcal{B},\mu)$ 
The system has the \textit{Kolmogorov property}, or simply the \textit{$K$-property}, if there exists a sub-$\sigma$-algebra $\mathcal{K} \subset \mathcal{B}$ satisfying $\mathcal{K} \subset f\mathcal{K}$, $\bigvee\limits_{i=0}^\infty f^i\mathcal{K} = \mathcal{B}$, and $\bigcap\limits_{i=0}^\infty f^{-i}\mathcal{K} = \{\emptyset,X\}$.

A measure-preserving flow $\mathcal{F} = \{f_t\}_{t \in \R}$ on $(X,\mathcal{B},\mu)$ has the \textit{$K$-property} if for every $t \neq 0$, the time-$t$ map $f_t \colon X\to X$ has the $K$-property.
\end{defn}

The $K$-property may be viewed as a mixing property that is weaker than Bernoulli but stronger than mixing of all orders. There are many alternative characterizations for the $K$-property. For instance, the system $(X,\mathcal{B},\mu,f)$ has the $K$-property if and only if $h_{\mu}(f,\xi)$ is positive for every non-trivial (mod 0) partition $\xi$ of $X$.


In order to apply \cite{call2019equilibrium}, we need to consider the flow $\F \times \F$ on the product space $T^1M \times T^1M$. Then the potential $\Phi \colon T^1M \times T^1M \to \R$ given by $\Phi(x,y):=\vphi(x)+\vphi(y)$ satisfies $P(\Phi) = 2P(\vphi)$.
Setting $$\wt{\la}_T(x,y) := \lt(x)\cdot\lt(y)$$
and $$L:=(\Sing\times T^1M)\cup (T^1M\times \Sing) \subset T^1M\times T^1M,$$ the following lemma is the analogue of Lemma \ref{charsing} for $\wt{\la}_T$.

\begin{lem}\label{nullla}
If $\wt{\la}_T(x,y)=0$ for all $T>0$, then $(x,y)\in L $.
\end{lem}
\begin{proof}
Without loss of generality, we may assume there exists $T_n\to\infty$ such that $\la_{T_n}(x)=0$ for all $n$. 
By Corollary \ref{incrsl}, $\wt{\la}_T$ is non-decreasing with respect to $T$, thus $\la_T(x)=0$ for all $T>0$. By Lemma \ref{charsing}, we have $x \in \Sing$.
\end{proof}

Given any $\F$-invariant measure $\nu \in \mathcal{M}(\F)$, we say $\nu$ is \textit{almost expansive at scale $\ep$} if $\nu(\text{NE}(\ep))=0$. From the flat strip theorem \cite{OSullivan:1976cf}, it follows that $\mu_\vphi$ is almost expansive at any sufficiently small scale $\ep>0$. Moreover, the pressure gap assumption $P(\vphi)>P(\Sing,\vphi)$ implies that any equilibrium state $\nu$ of $\Phi$ is product expansive; see \cite[Proposition 5.4]{call2019equilibrium}.	
Together with these observations, \cite[Theorem 6.5]{call2019equilibrium} states that the unique equilibrium state $\mu_{\vphi}$ from Theorem
\ref{thm: 1} has the $K$-property if ${\displaystyle \sup\Big\{ P_{\mu}(\Phi):\int\wt{\lambda}_{T}d\mu=0\Big\}<P(\Phi)=2P(\varphi)}$,
which we establish in the following lemma.

\begin{prop}
For any continuous potential $\varphi: T^1M\to\mathbb{R}$ with $P(\varphi) >P(\Sing,\varphi)$, there exists $T_0>0$ such that for any $T>T_0$,
$$\sup\Big\{P_{\mu}(\Phi):\int\wt{\lambda}_Td\mu=0\Big\}<P(\Phi)=2P(\varphi).$$
\end{prop}
\begin{proof}

Let $\nu$ be any $\F \times \F$-invariant measure with $\displaystyle\int\wt{\lambda}_Td\nu=0$ for all $T>0$. For any $n\in\mathbb{N}$, consider $E_n:=\{(x,y)\in\Reg\times\Reg: \wt{\la}_n(x,y)\neq 0\}$ and let $E:=\bigcup\limits_{n=1}^{\infty}E_n$.  

We claim that $\Reg\times\Reg$ is equal to $E$. In fact, if there exists $(x,y)\in (\Reg\times\Reg)\backslash E$, then $\wt{\la}_n(x,y)=0$, for all $n\in\mathbb{N}$. From Corollary \ref{incrsl}  we know that $\wt{\la}_T$ is non-decreasing with respect to $T$, thus $\wt{\la}_T(x,y)=0$ for all $T>0$. This is a contradiction because by Lemma \ref{nullla}, either $x$ or $y$ would have to lie in $\Sing$.

From the choice of $\nu$, it is clear that $\nu(E_n)=0$ for all $n \in \N$. Hence, $\nu(\Reg\times\Reg)=\nu(E)=0$ and $\nu(L)=1$. Consequently,
$$\mathcal{M}(L)=\bigcap_{T>0}\mathcal{M}_{\wt{\la}_T}(0).$$
Let $\ep:=P(\vphi)-P(\Sing,\vphi)>0$ be the size of the pressure gap. 
Similar to the proof of Proposition \ref{prop: CKP7.3}, there exists $T_0>0$ such that for any $T>T_0$ and $\mu\in \mathcal{M}_{\la_T}(0)$, we have
$$P_{\mu}(\Phi)<P(L,\Phi)+\ep/2.\eqno{(I)}$$
By Lemma 2.8 and Proposition 3.3 of \cite{call2019equilibrium}, we have 
$$P(L,\Phi)\leq P(\Sing, \varphi)+P(\varphi).\eqno{(II)}$$
By combining (I) and (II), we conclude that for any $T>T_0$, 
$$\sup\Big\{P_{\mu}(\Phi):\int\wt{\lambda}_Td\mu=0\Big\}\leq P(\Sing, \varphi)+P(\varphi)+\ep/2<2P(\varphi).$$
This completes the proof.
\end{proof}
\subsection{Other properties of $\mu_{\vphi}$}
Following the same proofs of \cite[Proposition 8.1 and 8.10]{chen2018unique},
we obtain the last two remaining properties of $\mu_{\vphi}$ claimed in Theorem \ref{thm: 2}. 
\begin{prop}
\label{prop:full_measure}$\mu_{\vphi}({\rm Reg})=1$ and $\mu$ is
fully supported. 
\end{prop}

\section{Proof of Theorem \ref{thm: 3}}\label{sec: proof for thm 3}
In this section, we prove Theorem \ref{thm: 3}.
We follow the proof of \cite[Theorem B]{burns2018unique} closely, while pointing out the necessary modifications required. 
We will make use of both functions $\l,\lt, \colon T^1M \to \R$ as well as the corresponding uniformity regular sets $\Reg(\eta)$ and $\Reg_T(\eta)$.

\subsection{Estimating singular orbits by regular orbits}

Fix a small $\eta_0>0$ such that $\Reg(\eta_0)$ has nonempty interior. 
From Corollary \ref{cor: uniform minimality}, there exists $R>0$ such that the $W^{s/u}_R(v)$ intersects $\Reg(\eta_0)$ for every $v \in T^1M$. By selecting any point from the intersection $W^{s/u}_R(v) \cap \Reg(\eta_0)$, we define maps 
$\Pi^\sigma \colon T^1M \to \Reg(\eta_0)$, $\sigma \in \{s,u\}$, so that $\Pi^\sigma(v) \in W^\sigma_R(v)$. Given $t>0$, we then define 
$\Pi_t \colon \Sing \to \Reg$ by
\begin{equation}\label{eq: Pi}
\Pi_t :=f_{-t} \circ \Pi^u \circ f_t \circ \Pi^s.
\end{equation}
\begin{prop}\cite[Theorem 8.1]{burns2018unique}\label{prop: 8.1 of bcft}
For every $\delta>0$ and $\eta \in (0,\eta_0)$, there exists $L>0$ such that for every $v \in \Sing$ and $t \geq 2L$, the image $w:= \Pi_t(v)$ has the following properties:
\begin{enumerate}
\item $w, f_tw \in\Reg(\eta)$;
\item $d_K(f_sw,\Sing) <\delta$ for all $s \in [L,t-L]$;
\item for every $s \in [L,t-L]$, $f_sw$ and $v$ lies in the same connected component of $B(\Sing,\delta):=\{w \in T^1M \colon d_K(w,\Sing)<\delta\}$.
\end{enumerate}
\end{prop}
\begin{figure}[H]\label{figure1}
\begin{center}\includegraphics[scale=0.5]{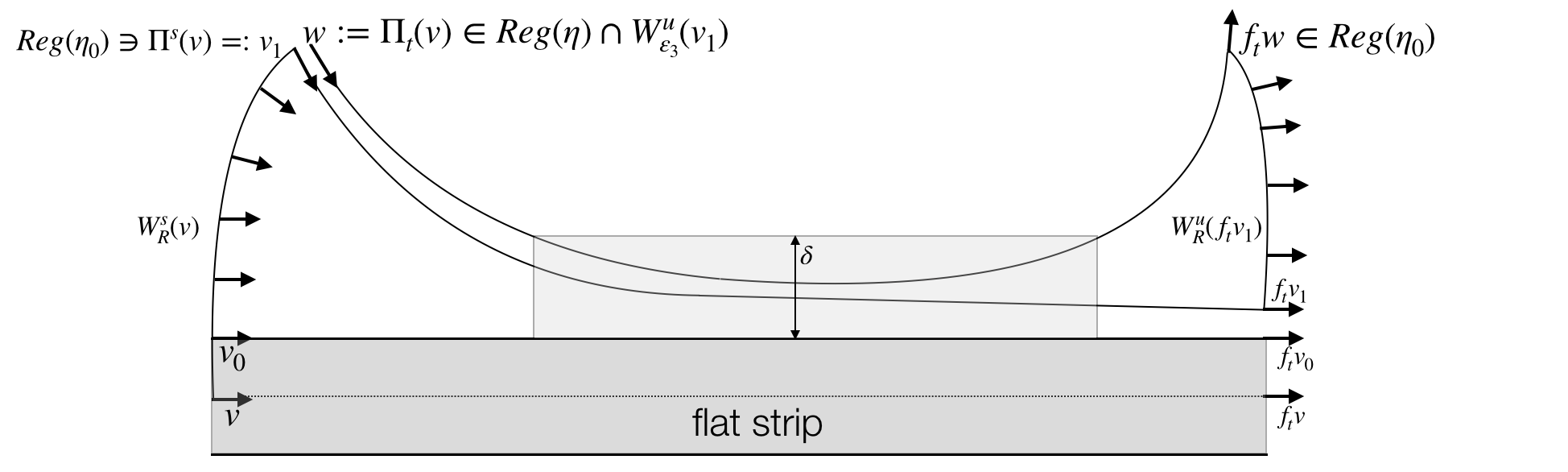}\end{center}
\caption{The approximation map $\Pi_t \colon \Sing \to \Reg$}
\end{figure}

\begin{rem}
Recall from Lemma \ref{charsing} that $\lambda_T(v) = 0$ for all $T>0$ if and only if $v \in \Sing$. However, $\lambda(f_tv)=0$ for all $t \in \R$ does not necessarily imply that $v$ belongs to $\Sing$. Hence, in order to control the distance of $f_sw$ to $\Sing$ for $s \in [L,t-L]$, we will make use of $\lambda_T$ for a suitably chosen $T$ even though $T$ does not show up as one of the parameters in the statement of the proposition.

However, unlike Theorem \ref{thm: 1} and \ref{thm: 2} where we predominatly made use of $\lt$, this proposition is the only place in the proof of Theorem \ref{thm: 3} where we make use of $\lt$; for rest of the proof appearing in Subsection \ref{subsec: proof of thm C}, we will primarily be making use of $\lambda$ instead.
\end{rem}

\begin{proof}[Proof of Theorem \ref{prop: 8.1 of bcft}]
Let $\delta>0$ and $\eta \in (0,\eta_0)$ be given.
For any $v\in \Sing$, let $v_1:=\Pi^s(v) \in \Reg(\eta_0) \cap W^s_R(v)$ be a point obtained from $\Pi^s$. Connect $v$ to $v_1$ by a directed line segment lying in $W^s(v)$, and denote the first intersection of the line segment with the boundary of $\Sing$ by $v_0$. Since $v_0 \in \Sing$, for any $T >0$ and $t \in \R$ we have $\lambda_T
(f_tv_0)=0$. Throughout the proof, refer to Figure \ref{figure1}. Note that Figure \ref{figure1} is a slight modification of \cite[Figure 8.1]{burns2018unique} in order to account for different notations. 

We begin by choosing $T_*,\eta_*>0$ from Lemma \ref{lem: 2} such that $$\lambda_{T_*}(v) \leq \eta_*\implies d_K(v,\Sing)<\delta.$$

From the uniform continuity of $\lambda_{T_*}>0$, there exists $\ep_1>0$ such that for any $a,b \in T^1M$,
\begin{equation}\label{eq: 1}
 \lambda_{T_*}(a) < \eta_*/3 \text{ and } b \in B(a,\ep_1) \implies \lambda_{T_*}(b) <\eta_*/2. 
\end{equation} 
From Lemma \ref{lem: 3}, there exists $\T_1>0$ such that for any $a \in \Reg_{T_*}(\eta_*/2)$ and any $t \geq \T_1$, we have $f_{-t}W^s_{\ep_1}(a) \supset W^s_R(f_{-t}a)$.\\\\
\textbf{Claim 1}: for any $t \geq \T_1$, we have $\lambda_{T_*}(f_tv_1) \leq \eta_*/2$.
\begin{proof}[Proof of Claim 1]
Suppose the claim does not hold; that is, $\lambda_{T_*}(f_tv_1) > \eta_*/2$ for some $t \geq \T_1$. Since $\lambda_{T_*}(f_tv_0) = 0$, we must have $d(f_tv_1, f_tv_0) >\ep_1$ from \eqref{eq: 1}. In particular, $W^s_{\ep_1}(f_tv_1)$ does not contain $f_tv_0$. 

On the other hand, since $t \geq \T_1$ and $\lambda_{T_*}(f_tv_1) >\eta_*/2$, $f_{-t}W^s_{\ep_1}(f_tv_1)$ contains $W^s_R(v_1)$. In particular, $f_{-t}W^s_{\ep_1}(f_tv_1)$ must contain $v_0$ because $v_0 \in W^s_R(v_1)$. However, this implies that $W^s_{\ep_1}(f_tv_1)$ contains $f_t(v_0)$, and this is a contradiction to the previous paragraph.
\end{proof}

Using uniform continuity of $\lambda_{T_*}$, we choose $\ep_2>0$ such that for any $a,b \in T^1M$,
\begin{equation}\label{eq: 2}
 \lambda_{T_*}(a) < \eta_*/2 \text{ and } b \in B(a,\ep_2) \implies \lambda_{T_*}(b) <\eta_*. 
\end{equation}
From Lemma \ref{lem: 3}, there exists $\T_2>0$ such that for any $t \geq \T_2$ and any $a \in \Reg_{T_*}(\eta_*)$, we have $f_tW^u_{\ep_2}(a)  \supset W^u_R(f_ta)$. 
Recalling that $w:=\Pi_t(v)$, we have the following claim whose proof is similar to the proof of Claim 1.\\\\
\textbf{Claim 2}: Suppose $t>\T_1+\T_2$ and $s \in [\T_1,t-\T_2]$. Then, $d_K(f_sw,\Sing)<\delta$.
\begin{proof}[Proof of Claim 2]
From the choice of $T_*,\eta_*>0$, it suffices to show that $\lambda_{T_*}(f_sw) \leq \eta_*$ for all $s \in [\T_1,t-\T_2]$.  

Suppose for the contrary that $\lambda_{T_*}(f_{s_0}w) >\eta_*$ for some for $s_0 \in [\T_1,t-\T_2]$. Since $\lambda_{T_*}(f_sv_1) \leq \eta_*/2$ for any $s \geq \T_1$, we have $d_K(f_{s_0}w,f_{s_0}v_1)>\ep_2$ from \eqref{eq: 2}. In particular, $W^u_{\ep_2}(f_{s_0}w)$ does not contain $f_{s_0}v_1$. 

On the other hand, since $t-s_0 \geq \T_2$ and $\lambda_{T_*}(f_{s_0}w) >\eta_*$, it follows that $f_{t-s_0}W^u_{\ep_2}(f_{s_0}w)$ contains $W^u_R(f_tw)$ which then contains $f_tv_1$ from the definition \eqref{eq: Pi} of $\Pi_t$. However, this is a contradiction to the previous paragraph that $f_{s_0}v_1$ does not belong to $W^u_{\ep_2}(f_{s_0}w)$.
\end{proof}	

Lastly, from the uniform continuity of $\lambda$, we choose a small $\ep_3>0$ such that for any $a,b\in T^1M$,
$$
\lambda(a) \geq \eta_0 \text{ and }b\in B(a,\ep_3) \implies \lambda(b) \geq \eta.
$$
Then by taking the compact subset $X$ to be $\Reg(\eta_0)$ from Lemma \ref{lem: 3}, we obtain $\T_3>0$ such that $f_tW^u_{\ep_3}(v_1)$ contains $W^u_R(f_tv_1)$ for any $t \geq \T_3$. From such choices of $\ep_3, \T_3$ and the fact that $v_1 \in \Reg(\eta_0)$, we ensure that $w:=\Pi_t(v) \in \Reg(\eta)$ whenever $t \geq \T_3$.
Putting everything together, the proposition follows by setting $L := \max\{\T_3/2,\max(\T_1,\T_2)\}$. 
\end{proof}

The following proposition shows that given any $(t,2\ep)$-separate subset $E$ of $\Sing$, the cardinality of the intersection between $\Pi_t(E)$ and any $(t,\ep)$-Bowen ball is uniformly bounded. The proof mostly follows that of \cite[Proposition 8.2]{burns2018unique} up until the end where small modification has to be made in order to account for the slightly weaker geometric features of manifolds without focal points.
\begin{prop}\cite[Proposition 8.2]{burns2018unique}\label{prop: 8.2 of BCFT}
For every $\ep>0$, there exists $C>0$ such that if $E_t \subset \Sing$ is a $(t,2\ep)$-separated set for some $t>0$, then for every $w \in T^1M$, we have $\#\{v \in E_t \colon d_t(w,\Pi_tv)<\ep\} \leq C$.
\end{prop}

\begin{proof}
Let $\wt{M}$ be the universal cover of $M$ and $B \subset \wt{M}$ a fundamental domain. Define $\wt{\Pi}^{s,u}$ and $\wt{\Pi}_t$ by lifting $\Pi^{s,u}$ and $\Pi_t$ to the universal cover. Let $\wt{d}$ be the lift of $d$ to $\wt{M}$, and $\wt{d}_t$ be the lift of $d_t$ to $T^1\wt{M}$. For every $\wt{v}\in T^1\wt{M}$, we have
\begin{equation}\label{eq: 01}
\wt{d}(\pi \wt{v},\pi \wt{\Pi}_t\wt{v}) \leq \wt{d}(\pi \wt{v},\pi \wt{\Pi}^s\wt{v})+\wt{d}(\pi\wt{\Pi}^s\wt{v},\pi\wt{\Pi}_t\wt{v}) \leq d^s(v,\Pi^sv)+d^u(\Pi^sv,\Pi_t v) \leq 2R,
\end{equation}
where $\pi$ is the projection from $T^1\wt{M}$ onto $\wt{M}$. By the analogous calculation, we have 
\begin{equation}\label{eq: 02}
\wt{d}(\pi f_t\wt{v},\pi f_t(\wt{\Pi}_t\wt{v}))<2R.
\end{equation}
Given $v \in T^1M$, let $\wt{v}_B \in T^1\wt{M}$ be the lift of $v$ such that $\pi \wt{v}_B \in B$. Then \eqref{eq: 01} gives $\displaystyle \pi\wt{\Pi}_t\wt{v}_B \in A_{2R}:= \bigcup_{x\in B} B_{\wt{d}}(x,2R)$.

Fix $\ep>0$ and let $\Gamma = \Gamma_{2R+\ep}:=\{g \in \pi_1(M) \colon gB \cap A_{2R+\ep} \neq \emptyset\}$. Note $\# \Gamma <\infty$ as $\overline{B}$ is compact. For $t >0$, let $E_t \subset \Sing$ be any $(t,2\ep)$-separated set, and fix an arbitrary $w \in T^1M$. We define
$$E^{w,\ep}_t:= \{v \in E_t \colon d_t(w,\Pi_t v)<\ep\}.$$

Since $d_t(w,\Pi_tv)<\ep$, there exists a lift $\wt{w}$ of $w$ such that $\wt{d}_t(\wt{w},\wt{\Pi}_t\wt{v}_B)<\ep$. As $\pi \wt{\Pi}_t\wt{v}_B \in A_{2R}$, it follows that $\wt{w} \in A_{2R+\ep}$, and thus $\pi \wt{w} \in gB$ for some $g\in \Gamma$. Thus,
$E^{w,\ep}_t = \bigcup\limits_{g \in \Gamma}E^g_t,$ where
$$E^g_t := \{v \in E^{w,\ep}_t \colon \wt{d}_t(\wt{w}_g,\wt{\Pi}_t\wt{v})<\ep \text{ where } \wt{w}_g \text{ is the lift of } w \text{ in }gB\}.$$

Let $X \subset B$ and $Y \subset A_{2R+\ep}$ be finite $\ep/2$-dense sets. We will show that $\# E^{w,\ep}_t \leq  (\#\Gamma)(\# X) (\#Y)$. 

For any $g \in \Gamma$ and $v \in E^g_t$, we approximate $\wt{v}_B$ and $f_t\wt{v}_B$ using $X$ and $Y$. Since $\pi\wt{v} \in B$, there exists $x=x(v) \in X$ such that $\wt{d}(x,\pi\wt{v}_B)<\ep/2$. For $f_t\wt{v}_B$, we use $f_t\wt{w}$ as the reference point. There exists a unique $h \in \pi_1(M)$ such that $\pi f_t\wt{w} \in hB$. Since $\wt{d}_t(\wt{w},\wt{\Pi}_t\wt{v}_B)<\ep$, we have $\wt{d}(hB,\pi f_t(\wt{\Pi}_t\wt{v}_B))<\ep$. This implies that $\pi f_t \wt{v}_B \in h(A_{2R+\ep})$ because $\wt{d}(\pi f_t\wt{v}_B,\pi f_t(\wt{\Pi}_t\wt{v}_B))<2R$ from \eqref{eq: 02}. In particular, there exists some $y=y(v) \in Y$ such that $\wt{d}(\pi f_t\wt{v}_B,h(y))<\ep/2$.

Now we show that the map $x \times y \colon E^g_t \to X \times Y$ is injective. Given any $v_1,v_2 \in E^g_t$ with $(x\times y )(v_1) = (x\times y )(v_2)$, define $\rho(s):=\widetilde{d}(\gamma_{\widetilde{v}_1}(s),\gamma_{\widetilde{v}_2}(s))$ for $s \in [0,t]$. From $x(v_1) = x(v_2)=:x_0$, we have $\rho(0)=\widetilde{d}(\gamma_{\widetilde{v}_1}(0),\gamma_{\widetilde{v}_2}(0)) < \ep/2 +\ep/2 =\ep$ by applying the triangle inequality pivoted at $\gamma_{\widetilde{x}_0}(0)$. Similarly, $\rho(t)<\ep$ from $y(v_1) = y(v_2)$. While \cite{burns2018unique} uses the convexity of the function $\rho$, a property coming from the geometry of nonpositively curved manifolds, to conclude that $\rho$ attains its maximum value at an endpoint, such property of $\rho$ is not available for manifolds without focal points. Instead, we use \cite[Proposition 2.8]{katok1982entropy} to conclude that
$$\rho(s) \leq \rho(0)+\rho(t) <2\ep \text{ for all }s \in [0,t],$$
and hence, $d_t(v_1,v_2) \leq 2\ep$. Since $E_t^g$ is $(t,2\ep)$-separated, this gives $v_1 = v_2$. Injectivity shows that $\#E^g_t \leq (\# X) (\#Y)$ for each $g \in \Gamma$, and this proves the proposition with $C := (\#\Gamma)(\# X) (\#Y)$.
\end{proof}

\subsection{Proof of Theorem \ref{thm: 3}}\label{subsec: proof of thm C}

With Proposition \ref{prop: 8.1 of bcft} and \ref{prop: 8.2 of BCFT}, the rest of the proof can be completed as in \cite{burns2018unique}.
For completeness, we provide a brief sketch of the proof by closely following \cite{burns2018unique}: we will sketch the existence of constants $\tau,Q$ and $\T$ independent of $\alpha$ such that 
\begin{equation}\label{eq: final}
P(\vphi) \geq -\frac{\alpha}{\tau}\log \alpha - \frac{\alpha Q}{\tau}+\Big(1 - \frac{\alpha \T}{\tau}\Big)P(\Sing,\vphi).
\end{equation}
Then by choosing $\alpha \in (0,e^{-Q-\T P(\Sing,\vphi)})$, we will obtain the required pressure gap $P(\vphi) >P(\Sing, \vphi)$. Note the inequality \eqref{eq: final} is slightly different from its corresponding inequality \cite[(8.16)]{burns2018unique} due to a small error in \cite[Proposition 8.7]{burns2018unique}. However, such error that can be easily amended; see \eqref{eq: 10}.

We begin with the following general lemma from \cite{climenhaga2016unique}.
\begin{lem}\label{lem: from CT}
Let $(X,\F)$ be a continuous flow on a compact metric space, let $\vphi \colon X \to \R$ be continuous, and let $\ep>0$. Then for all $t >0$,
\begin{equation}\label{eq: 4}
\sup \Big\{ \sum\limits_{x\in E} \exp\Big({\sup\limits_{y \in B_t(x,\ep)} \Phi(y,t)} \Big) \colon E \subset X\text{ is }(t,\ep)\text{-separated} \Big\} \geq e^{tP(X,2\ep,\vphi)}.
\end{equation}
\end{lem}

Under the assumption of Theorem \ref{thm: 3} that $\vphi$ is locally constant on a neighborhood of $\Sing$, for any sufficiently small $\ep>0$ and any $x \in \Sing$ and $t>0$, the value of $\Phi(y,t)$ is constant for any $y \in B_t(x,\ep)$. In particular, the left hand side of \eqref{eq: 4} is equal to $\Lambda(\Sing,\vphi, \ep,t)$ when applied to $X = \Sing$. From the entropy-expansivity of the geodesic flows on manifolds with no focal points \cite{liu2016entropy}, we have $P(\Sing,\vphi) = P(\Sing,2\ep, \vphi)$ for any sufficiently small $\ep$. Thus, for sufficiently small $\ep>0$, Lemma \ref{lem: from CT} gives
\begin{equation}\label{eq: 5}
\Lambda(\Sing,\vphi,\ep,t) \geq e^{tP(\Sing,\vphi)}.
\end{equation}
We now choose the constants from Proposition \ref{prop: 8.1 of bcft}. Recall that we fixed $\eta_0$ early in the process of defining the map $\Pi_t$. Now fix $\eta \in (0,\eta_0)$ and $\delta>0$ sufficiently small such that
\begin{enumerate}[label=(\roman*)]
\item\label{eq: i} $\vphi$ is locally constant on $B(\Sing,2\delta)$, and
\item\label{eq: ii} $d(v,\Sing)<2\delta \implies \lambda(v)<\eta$.
\end{enumerate}
Since $\l|_\Sing \equiv 0$, such choices of $\eta,\delta$ can be made by first choosing $\delta>0$ and $\eta \in (0,\eta_0)$ that satisfy \ref{eq: ii} using the uniform continuity of $\lambda$ and then by decreasing $\delta$ to satisfy \ref{eq: i} if necessary. 

Denoting the components of $\Sing$ by $U_1,\ldots,U_k$, and the constant value of $\vphi$ on $U_i$ by $\Phi_i$, \ref{eq: i} and \eqref{eq: 5} together produce a $(t,2\delta)$-separated set $E_t \subset \Sing$ such that
\begin{equation}\label{eq: 6}
\sum\limits_{i=1}^k e^{t\Phi_i} \# (E_t \cap U_i) \geq  e^{t P(\Sing,\vphi)}.
\end{equation}
Let $L=L(\delta, \eta)$ be from Proposition \ref{prop: 8.1 of bcft}. From the third statement of Proposition \ref{prop: 8.1 of bcft}, for any $t \geq 2L$, any $v \in E_t \cap U_i$, and any $s \in [L,t-L]$, we have $d(f_s\Pi_tv,U_i)<\delta$. From \ref{eq: i}, for any $u \in B_t(\Pi_tv,\delta)$, we have
\begin{equation}\label{eq: 7}
\int_0^t \vphi(f_su)ds \geq (t-2L)\Phi_i - 2L\|\vphi\|\geq  t\Phi_i -4L\|\vphi\|.
\end{equation}

For each $1 \leq i \leq k$, let $E_{t,i}':=\Pi_t(E_t \cap U_i)$ and $E_{t,i}''$ be a maximal $(t,\delta)$-separated subset of $E_{t,i}'$. From Proposition \ref{prop: 8.2 of BCFT}, there exists $C>0$ such that for any $w \in T^1M$, the cardinality of the intersection $E_{t,i}' \cap B_t(w,\delta)$ is bounded above by $C$. Since $E_{t,i}''$ is $(t,\delta)$-spanning in $E_{t,i}'$, by varying $w$ over the elements of $E_{t,i}''$ we have $\# E_{t,i}'' \geq C^{-1}\# E_{t,i}'$. 

To summarize the construction so far, given any $t \geq 2L$, we construct a $(t,\delta)$-separated set $E_{t}'' := \bigcup\limits_{i=1}^k E_{t,i}'' \subset E_t'$ such that for $\beta:= C^{-1}e^{-4L\|\vphi\|}$, $E_{t}''$ satisfies
\begin{equation}\label{eq: 8}
\sum\limits_{w \in E_t''} e^{\inf_{u \in B_t(w,\delta)}\int_0^t \vphi(f_su)du } \geq \beta e^{tP(\Sing,\vphi)}
\end{equation}
from \eqref{eq: 6} and \eqref{eq: 7}; see \cite[Lemma 8.4]{burns2018unique} for details.

Since $\Reg(\eta)$ is a compact subset of $\Reg$, from Lemma \ref{charsing} there exist $T_*,\eta_*>0$ such that $\Reg(\eta) \subset \Reg_{T_*}(\eta_*)$. 
Let $\T>0$ be the transition time for the specification property obtained from Proposition \ref{prop: specification} applied to $\cal{C}_{T_*}(\eta_*)$ at scale $\rho = \delta/3$; see Definition \ref{defn: specification}. Then for any given number of orbit segments $\{(v_i,t_i)\}_{i=1}^n$ with $v_i,f_{t_i}v_i \in \Reg(\eta)$ and any $\{T_i\}_{i=1}^n$ with $T_1=0$ and $T_{i}-T_{i-1} \geq t_i+\T$, there exists $w \in T^1M$ such that for each $1\leq i \leq n$,
$$f_{T_i}(w) \in B_{t_i}(v_i,\delta/3)).$$

Let $\alpha\in (0,1)$ be small and $N \in \N$ such that $\alpha N \in \N$. Set $\tau := 2L+\T$  and consider 
$$\mathcal{A} = \{\tau,2\tau,\ldots,(N-1)\tau\} \subset [0,N\tau].$$
Let $\mathbb{J}_N^\alpha:=\{J \subset \mathcal{A} \colon \#J = \alpha N-1\}$ be the set of all $(\alpha N-1)$-subsets of $\mathcal{A}$. 
We call each element of $\in \mathbb{J}^\alpha_N$ an \textit{itinerary}. 

For any given itinerary $$J=\{(N_1 \tau,\ldots,N_{\alpha N-1}\tau) \colon 1\leq N_1<\ldots <N_{\alpha N-1} \leq N-1\},$$ let 
$t_i = (N_i - N_{i-1})\tau - \T$ for $1 \leq i \leq \alpha N$ where $N_0 = 0$ and $N_{\alpha N} = N$. As $t_i \geq \tau-\T \geq  2L$, there exists a $(t_i,\delta)$-separated set $E_{t_i}''$ satisfying \eqref{eq: 8}. Note that the endpoints $v$ and $f_{t_i}v$ of any orbit segment $(v,t_{i})$ with $v \in E_{t_i}''$ belong to $\Reg(\eta)$ from Proposition \ref{prop: 8.1 of bcft}. Hence, for any $\textbf{v}:=(v_1,\ldots,v_{\alpha N}) \in \prod\limits_{i=1}^{\alpha N} E_{t_i}''$, there exists $G(\textbf{v}) \in T^1M$ from the specification property such that  
\begin{equation}\label{eq: 9}
f_{N_{i-1}\tau}G(\textbf{v}) \in B_{t_i}(v_i,\delta/3) \text{ for each }1 \leq i \leq \alpha N.
\end{equation}
It is shown in \cite[Lemma 8.5]{burns2018unique} that the image $\chi_J$ of $G \colon \prod\limits_{i=1}^{\alpha N} E_{t_i}'' \to T^1M$ is $(N\tau,\delta/3)$-separated. From \cite[Lemma 8.6]{burns2018unique}, Proposition \ref{prop: 8.1 of bcft} and \ref{eq: ii} imply that for any $w \in \chi_J$, we have
\begin{align}\label{eq: 8.5}
\begin{split}
d(f_tw,\Sing) &> 5\delta/3 \text{ for all } t\in J,\\ 
d(f_tw,\Sing) &< 4\delta/3 \text{ for all } t \in \mathcal{A} \setminus J.
\end{split}
\end{align}
Moreover, for any $G(\mathbf{v})$ with $\mathbf{v}=(v_1,\ldots,v_{\alpha N})$, the integral $\int_{N_{i-1}\tau}^{N_{i-1}\tau+t_{i}} \vphi(f_sw)ds$ is bounded below by $\inf_{u \in B_{t_i}(v_i,\delta /3)} \int _0^{t_i} \vphi(f_su)ds$ from \eqref{eq: 9} and $\int^{N_{i}\tau}_{N_{i-1}\tau+t_{i}} \vphi(f_sw)ds$ is trivially bounded below by $-\T\|\vphi\|$. Summing $\int_0^{N\tau} \vphi(f_sw)ds$ over all $w \in \chi_J$, \eqref{eq: 8} implies that the $(N\tau,\delta/3)$-separated subset $\chi_J$ satisfies 
\begin{equation}\label{eq: 10}
\sum\limits_{w \in \chi_J} e^{\int_0^{N\tau} \vphi(f_sw)ds} \geq e^{-\alpha N \T\|\vphi\|}\prod_{i=1}^{\alpha N} \beta e^{t_{i}P(\Sing,\vphi)} \geq e^{-\alpha NQ}e^{(N\tau - \alpha N \T)P(\Sing,\vphi)},
\end{equation}
where $\beta: = C^{-1}e^{-4L\|\vphi\|}$ is from \eqref{eq: 8} and $Q:=\T\|\vphi\|-\log \beta$. We remark that the constants $Q$ and $\T$ do not depend $\alpha$. 

Note that for any $J \neq J' \in \mathbb{J}_N^\alpha$ and any $v \in \chi_J$ and $w \in \chi_{J'}$, there exists at least one $t \in \mathcal{A}$ such that one of $f_tv$ and $f_tw$ is $4\delta/3$-close to Sing while another is at least $5\delta/3$-away from Sing from \eqref{eq: 8.5}. In particular, $d_{N\tau}(v,w)\geq \delta/3$. This implies that the union $\bigcup\limits_{J \in \mathbb{J}_N^\alpha} \chi_J$ of each $(N\tau,\delta/3)$-separated set $\chi_J$ is still $(N\tau,\delta/3)$-separated. Moreover, it follows from \eqref{eq: 10} that
\begin{equation}\label{eq: 11}
\sum\limits_{w \in \bigcup\limits_{J \in \mathbb{J}_N^\alpha} \chi_J} e^{\int_0^{N\tau} \vphi(f_sw)dw} \geq {N-1 \choose \alpha N-1} e^{-\alpha NQ}e^{(N\tau - \alpha N \T)P(\Sing,\vphi)}
\end{equation}
because $\#\mathbb{J}_N^\alpha = {N-1 \choose \alpha N-1}$. 

As $\bigcup\limits_{J \in \mathbb{J}_N^\alpha} \chi_J$ is $(N\tau,\delta/3)$-separated, the outcome of taking a logarithm of the left hand side of $\eqref{eq: 11}$ followed by dividing by $N\tau$ and taking the limit as $N$ approaches $\infty$ is bounded above by $P(\vphi)$. Using the fact that $\frac{N-k}{\alpha N - k }\geq \frac{1}{\alpha}$ for all $1 \leq k \leq \alpha N$, we have ${N-1 \choose \alpha N - 1} \geq \alpha e^{(-\alpha \log \alpha)N}$. Hence, \eqref{eq: 11} implies that
$$P(\vphi) \geq - \frac{\alpha}{\tau} \log \alpha -\frac{\alpha Q}{\tau} +\Big(1- \frac{\alpha \T}{\tau}\Big)P(\Sing,\vphi).$$
In particular, we obtain \eqref{eq: final} and establish required pressure gap \eqref{eq: pressure gap} by choosing $\alpha$ sufficiently small in $(0,e^{-Q-\T P(\Sing,\vphi)})$.

\bibliographystyle{amsalpha}
\bibliography{BIB}

\providecommand{\bysame}{\leavevmode\hbox to3em{\hrulefill}\thinspace}
\providecommand{\MR}{\relax\ifhmode\unskip\space\fi MR }
\providecommand{\MRhref}[2]{%
  \href{http://www.ams.org/mathscinet-getitem?mr=#1}{#2}
}
\providecommand{\href}[2]{#2}
\begin{thebibliography}{BCFT18}

\bibitem[Bab02]{Babillot:2002bq}
Martine Babillot, \emph{On the mixing property for hyperbolic systems}, Israel
  J. Math. \textbf{129} (2002), 61--76. \MR{1910932}

\bibitem[Bal82]{ballmann1982axial}
Werner Ballmann, \emph{Axial isometries of manifolds of non-positive
  curvature}, Mathematische Annalen \textbf{259} (1982), no.~1, 131--144.

\bibitem[BCFT18]{burns2018unique}
Keith Burns, Vaughn Climenhaga, Todd Fisher, and Daniel~J Thompson,
  \emph{Unique equilibrium states for geodesic flows in nonpositive curvature},
  Geometric and Functional Analysis \textbf{28} (2018), no.~5, 1209--1259.

\bibitem[Bow74]{bowen1974some}
Rufus Bowen, \emph{Some systems with unique equilibrium states}, Theory of
  computing systems \textbf{8} (1974), no.~3, 193--202.

\bibitem[Bur83]{Burns:1983dw}
Keith Burns, \emph{Hyperbolic behaviour of geodesic flows on manifolds with no
  focal points}, Ergodic Theory Dynam. Systems \textbf{3} (1983), no.~1, 1--12.
  \MR{743026}

\bibitem[CFT18]{Climenhaga:2015wf}
Vaughn Climenhaga, Todd Fisher, and Daniel~J. Thompson, \emph{Unique
  equilibrium states for {B}onatti-{V}iana diffeomorphisms}, Nonlinearity
  \textbf{31} (2018), no.~6, 2532--2570. \MR{3816730}

\bibitem[CFT19]{climenhaga2019equilibrium}
Vaughn Climenhaga, Todd Fisher, and Daniel~J Thompson, \emph{Equilibrium states
  for man{\'e} diffeomorphisms}, Ergodic Theory and Dynamical Systems
  \textbf{39} (2019), no.~9, 2433--2455.

\bibitem[CKP18]{chen2018unique}
Dong Chen, Lien-Yung Kao, and Kiho Park, \emph{Unique equilibrium states for
  geodesic flows over surfaces without focal points}, to appear in Nonlinearity
  (2018).

\bibitem[CKW19]{climenhaga2019uniqueness}
Vaughn Climenhaga, Gerhard Knieper, and Khadim War, \emph{{Uniqueness of the
  measure of maximal entropy for geodesic flows on certain manifolds without
  conjugate points}}, arXiv preprint arXiv:1903.09831v1 (2019).

\bibitem[CT16]{climenhaga2016unique}
Vaughn Climenhaga and Daniel~J Thompson, \emph{Unique equilibrium states for
  flows and homeomorphisms with non-uniform structure}, Advances in Mathematics
  \textbf{303} (2016), 745--799.

\bibitem[CT19]{call2019equilibrium}
Benjamin Call and Daniel~J Thompson, \emph{Equilibrium states for products of
  flows and the mixing properties of rank 1 geodesic flows}, arXiv preprint
  arXiv:1906.09315 (2019).

\bibitem[dC13]{doCarmo:2013tg}
Manfredo~P do~Carmo, \emph{{Riemannian Geometry}}, Birkh{\"a}user, January
  2013.

\bibitem[Ebe73a]{Eberlein1973II}
Patrick Eberlein, \emph{Geodesic flows on negatively curved manifolds ii},
  Transactions of the American Mathematical Society \textbf{178} (1973),
  57--82.

\bibitem[Ebe73b]{Eberlein:1973hu}
\bysame, \emph{{When is a geodesic flow of Anosov type? I}}, J. Differential
  Geometry \textbf{8} (1973), no.~3, 437--463. \MR{0380891}

\bibitem[Ebe96]{eberlein1996geometry}
\bysame, \emph{Geometry of nonpositively curved manifolds}, University of
  Chicago Press, 1996.

\bibitem[EO76]{eschenburg1976growth}
Jost-Hinrich Eschenburg and John~J O'Sullivan, \emph{Growth of jacobi fields
  and divergence of geodesics}, Mathematische Zeitschrift \textbf{150} (1976),
  no.~3, 221--237.

\bibitem[Esc77]{Eschenburg:1977kn}
Jost-Hinrich Eschenburg, \emph{Horospheres and the stable part of the geodesic
  flow}, Math. Z. \textbf{153} (1977), no.~3, 237--251. \MR{0440605}

\bibitem[Fra77]{Franco:1977jy}
Ernesto Franco, \emph{Flows with unique equilibrium states}, Amer. J. Math.
  \textbf{99} (1977), no.~3, 486--514. \MR{0442193}

\bibitem[GR19]{Gelfert:2017tx}
Katrin Gelfert and Rafael~O. Ruggiero, \emph{Geodesic flows modelled by
  expansive flows}, Proc. Edinb. Math. Soc. (2) \textbf{62} (2019), no.~1,
  61--95. \MR{3938818}

\bibitem[GS14]{Gelfert:2014hn}
Katrin Gelfert and Barbara Schapira, \emph{Pressures for geodesic flows of rank
  one manifolds}, Nonlinearity \textbf{27} (2014), no.~7, 1575--1594.
  \MR{3225873}

\bibitem[Hur86]{Hurley:1986km}
Donal Hurley, \emph{Ergodicity of the geodesic flow on rank one manifolds
  without focal points}, Proc. Roy. Irish Acad. Sect. A \textbf{86} (1986),
  no.~1, 19--30. \MR{865098}

\bibitem[Kat82]{katok1982entropy}
Anatole Katok, \emph{Entropy and closed geodesies}, Ergodic theory and
  dynamical Systems \textbf{2} (1982), no.~3-4, 339--365.

\bibitem[KH97]{katok1997introduction}
Anatole Katok and Boris Hasselblatt, \emph{Introduction to the modern theory of
  dynamical systems}, vol.~54, Cambridge university press, 1997.

\bibitem[Kni98]{knieper1998uniqueness}
Gerhard Knieper, \emph{The uniqueness of the measure of maximal entropy for
  geodesic flows on rank 1 manifolds}, Annals of mathematics (1998), 291--314.

\bibitem[Led77]{ledrappier1977mesures}
Fran{\c{c}}ois Ledrappier, \emph{Mesures d'{\'e}quilibre d'entropie
  compl{\`e}tement positive}, Ast{\'e}risque \textbf{50} (1977), 251--272.

\bibitem[LLW18]{Liu:2018ud}
Fei Liu, Xiaokai Liu, and Fang Wang, \emph{{On the mixing and Bernoulli
  properties for geodesic flows on rank 1 manifolds without focal points}},
  arXiv.org (2018).

\bibitem[LW16]{liu2016entropy}
Fei Liu and Fang Wang, \emph{Entropy-expansiveness of geodesic flows on closed
  manifolds without conjugate points}, Acta Mathematica Sinica, English Series
  \textbf{32} (2016), no.~4, 507--520.

\bibitem[LWW20]{liu2018patterson}
Fei Liu, Fang Wang, and Weisheng Wu, \emph{On the patterson-sullivan measure
  for geodesic flows on rank $1 $ manifolds without focal points}, Discrete and
  Continuous Dynamical Systems \textbf{40} (2020), no.~3, 1517 -- 1554.

\bibitem[O'S76]{OSullivan:1976cf}
John~J. O'Sullivan, \emph{Riemannian manifolds without focal points}, J.
  Differential Geometry \textbf{11} (1976), no.~3, 321--333. \MR{0431036}

\bibitem[Par88]{parry1988equilibrium}
William Parry, \emph{Equilibrium states and weighted uniform distribution of
  closed orbits}, Dynamical Systems, Springer, 1988, pp.~617--625.

\bibitem[Pes77]{Pesin:1977vi}
Ja.~B. Pesin, \emph{Geodesic flows in closed {R}iemannian manifolds without
  focal points}, Izv. Akad. Nauk SSSR Ser. Mat. \textbf{41} (1977), no.~6,
  1252--1288, 1447. \MR{0488169}

\bibitem[PP90]{Parry:1990tn}
William Parry and Mark Pollicott, \emph{{Zeta functions and the periodic orbit
  structure of hyperbolic dynamics}}, Ast\'erisque \textbf{187-188} (1990),
  1--268.

\bibitem[Wal00]{walters2000introduction}
Peter Walters, \emph{An introduction to ergodic theory}, Graduate Texts in
  Mathematics, vol.~79, Springer-Verlag, 2000.

\end{thebibliography}

\end{document}